\documentclass[12pt,reqno]{amsart}

\usepackage{tikz}
\usetikzlibrary{positioning, arrows.meta}
\usepackage{epsfig}
\usepackage{amscd}
\usepackage[mathscr]{eucal}
\usepackage{amssymb}
\usepackage{amsxtra}
\usepackage{amsmath}
\usepackage[all]{xy}
\usepackage{mathtools}
\usepackage[pdfencoding=auto]{hyperref}
\usepackage{bookmark}
\usepackage{caption}
\usepackage{colortbl}
\usepackage{arydshln}

\DeclarePairedDelimiter\abs{\lvert}{\rvert}

\theoremstyle{plain}
\newtheorem{mainthm}{Theorem} 
\newtheorem{thm}{Theorem}[section] 
\newtheorem{cor}[thm]{Corollary} 
\newtheorem{prop}[thm]{Proposition} 

\newtheorem*{conj}{Conjecture}

\theoremstyle{definition}
\newtheorem{dfn}[thm]{Definition} 

\theoremstyle{remark}
\newtheorem{rem}[thm]{Remark} 

\theoremstyle{definition} 
\newtheorem{exmp}{Example}[section] 

\begin{document}

\bigskip

\title[Edwards' Speculation and a New Var. Method for Zeros $Z$-Function]{On Edwards' Speculation and a New Variational Method for the Zeros of the $Z$-Function} 
\date{\today}

\author{Yochay Jerby}

\address{Yochay Jerby, Faculty of Sciences, Holon Institute of Technology, Holon, 5810201, Israel}
\email{yochayj@hit.ac.il}


%
%
% ...

\begin{abstract}

In his foundational book, Edwards introduced a unique "speculation" regarding the possible theoretical origins of the Riemann Hypothesis, based on the properties of the Riemann-Siegel formula.  Essentially Edwards asks whether one can find a method to transition from zeros of $Z_0(t)=cos(\theta(t))$, where $\theta(t)$ is Riemann-Siegel theta function, to zeros of $Z(t)$, the Hardy $Z$-function.  However, when applied directly to the classical Riemann-Siegel formula, it faces significant obstacles in forming a robust plausibility argument for the Riemann Hypothesis. 

In a recent work, we introduced an alternative to the Riemann-Siegel formula that utilizes series acceleration techniques. In this paper, we explore Edwards' speculation through the lens of our accelerated approach, which avoids many of the challenges encountered in the classical case. Our approach leads to the description of a novel variational framework for relating zeros of $Z_0(t)$ to zeros of $Z(t)$ through paths in a high-dimensional parameter space $\mathcal{Z}_N$, recasting the RH as a modern non-linear optimization problem.
\end{abstract}

\maketitle
%
%

\section{Introduction}
\label{s:s3}

 The Riemann zeta function is given by $\zeta(s) = \sum_{n=1}^{\infty} n^{-s}$ in the range $Re(s)>1$. In 1859, Bernhard Riemann published a work of monumental importance \cite{R}, whose starting point is the analytic extension of $\zeta(s)$ to a meromorphic function with a single pole at $s=1$. This allowed him to describe 
various intricate relationships between the distribution of prime numbers and the analytical properties of $\zeta(s)$. Within this  manuscript, Riemann posited what would become one of the most famous and enduring conjectures in mathematics: 

\begin{conj}[RH]
All non-trivial zeros of $\zeta(s)$ lie on the critical line $Re(s) = 1/2$.
\end{conj}

The RH, though only briefly mentioned in Riemann's manuscript, holds profound implications and its resolution would represent a significant breakthrough. It would not only provide a deeper understanding of the distribution of prime numbers but also substantiate crucial links across various domains of mathematics and physics with implications extending to number theory, algebra, cryptography, and quantum physics, see \cite{Borwein2008Riemann, RHP} and references therein.

However, it is noteworthy that within his ground breaking manuscript, Riemann did not describe any methods for the practical computation of $\zeta(s)$ for values in the complex domain, a fact that further contributed to the mystery of the conjecture. Indeed, one comes to the zen-like question: how can one locate the zeros of a function whose values one does not know how to efficiently compute altogether? 

For such reasons, in the early 20-th century, mathematicians such as G. H. Hardy \cite{Hardy1915} and E. Landau \cite{Landau1908} expressed the idea that the Riemann hypothesis was most likely the result of a heuristic guess rather than based on a concrete verification, not to mention on an insight into any conceptual theoretical reason as to why one should expect it to hold.

In the 1930s, C. L. Siegel \cite{Si} uncovered previously unpublished notes by Riemann, astonishingly showing that Riemann had actually developed sophisticated saddle point methods for the efficient numerical evaluation of the values of $\zeta(s)$ on the critical strip. Based on a numerical investigation of the properties of the Riemann-Siegel formula for the first zero of $\zeta(s)$, Edwards proposed an intriguing "speculation" in his seminal book \cite{E} regarding the possible theoretical origins of the Riemann Hypothesis. Recall that the Hardy $Z$-function is the real function defined by
\begin{equation} \label{eq:Hardy}
Z(t) = e^{i \theta(t)} \zeta \left ( \frac{1}{2} +it \right )
\end{equation}
where
\begin{equation} \label{eq:RS-theta}
\theta(t) = \text{arg} \left ( \Gamma \left ( \frac{1}{4} + \frac{i t}{2} \right ) \right ) -\frac{t}{2} \log(t).
 \end{equation} 
 is the Riemann-Siegel \(\theta\)-function \cite{E,I}. The Riemann Hypothesis is equivalent to the statement that all non-trivial zeros of the Hardy $Z$-function are real. We will refer to 
\begin{equation} 
Z_0(t):= cos(\theta(t)),
\end{equation}
as \emph{the core of $Z(t)$}. The zeros $t_n \in \mathbb{R}$ of the core $Z_0(t)$ are all real and are given as the unique solutions of 
\begin{equation} 
\theta(t_n) = \left ( n + \frac{1}{2} \right ) \pi, 
\end{equation} for any $n \in \mathbb{Z}$. Edwards' speculation is essentially the question of whether one can find a method to transition from zeros of $Z_0(t)$, which is the first term of the Riemann-Siegel formula, to zeros of $Z(t)$. 

 However, according to Edwards himself, while such a method might give insight to Riemanns' original motivation to conjecture the RH, when applied directly to the classical Riemann-Siegel formula, it faces significant obstacles in forming a robust plausibility argument for the Riemann Hypothesis. Rather, it further reaffirms its notorious difficulty rather than shed light on why it might hold true. 
 
 In recent works \cite{J,J4}, we introduced an alternative to the Riemann-Siegel formula via high-order sections based on series acceleration techniques. Our aim in this work is to investigate the ideas of Edwards' speculation applied to our accelerated formulas, rather than to the Riemann-Siegel formula. We show that our formula avoids many of the analytical challenges encountered in the classical Riemann-Siegel case, leading us to  introduce a new variational framework for the study of the zeros of the zeta function. 
 
 For any $N \in \mathbb{N}$, we introduce the space $\mathcal{Z}_N$, of generalized sections of Hardy's $Z$-function
\begin{equation} 
Z_N(t; \overline{a}) = cos(\theta(t))+ \sum_{k=1}^{N} \frac{a_k}{\sqrt{k+1} } cos ( \theta (t) - ln(k+1) t)
\end{equation}
where $\overline{a}=(a_1,...,a_N) \in \mathbb{R}^N$. Note that given a path $\gamma(r)$ in the space $\mathcal{Z}_N$ starting at $\gamma(0) = Z_0(t)$, one can continuously extend $t_n(r)$ to be a corresponding zero of $Z_N (t ; \gamma(r))$. Our main result in this work is the following:

\begin{mainthm}[Edwards' Speculation for High-Order Sections] \label{thm:1}
The Riemann Hypothesis holds if and only if, for any \(n \in \mathbb{Z}\), there exists a path \(\gamma(r)\) in \(\mathcal{Z}_{\left[ \frac{t_n}{2} \right]}\) from \(\overline{a}=\overline{0}\) to \(\overline{a}=\overline{1}\), along which \(t_n\) does not collide with its adjacent zeros \(t_{n\pm 1}\). That is, for which \(t_n(r) \neq t_{n \pm 1}(r)\) for all \(r \in [0,1]\).
\end{mainthm}

The proof of this result is partially based on previous results \cite{J,J4} on Spira's conjecture and high-order sections. In the subsequent sections, we further expand on the crucial differences between our Theorem \ref{thm:1} and the Riemann-Siegel formula setting. Mainly, we explain how, contrary to the classical case, Theorem \ref{thm:1} can indeed be considered as leading to the formulation of a meaningful plausibility argument for the RH.

\bigskip

The rest of this paper is organized as follows. In Section \ref{s:s2}, we recall the Riemann-Siegel formula and its role in numerical calculations of $Z(t)$. Section \ref{s:s3} introduces the formulation of Edwards' speculation, the core function $Z_0(t) = \cos(\theta(t))$ and its zeros. It outlines three foundational questions prompted by the speculation, each leading to in-depth discussions in the following Sections \ref{s:s4} to \ref{s:s6}

In Section \ref{s:s4}, we outline the experimental relation between the zeros of $Z(t)$ and those of the core $Z_0(t)$, central to Edwards' speculation. The shortcomings of using the Riemann-Siegel formula for the theoretical exploration of RH, particularly tracking the impact of terms in the Riemann-Siegel equation on the zeros locations, are discussed in Section \ref{s:s5}. Section \ref{s:s6} examines the limitations of classical numerical methods, such as the Newton method, for effectively transitioning from zeros of $Z_0(t)$ to those of $Z(t)$. Examples where the method fails to provide accurate correspondence between the zeros are presented, highlighting the necessity for more sophisticated approaches. 

In Section \ref{s:s7}, we present high-order sections as an alternative which avoids the shortcomings of the classical Riemann-Siegel formula outlined in Sections \ref{s:s4} to \ref{s:s6}. Building upon the high-order sections discussed in Section \ref{s:s7}, Section \ref{s:s8} develops the new variational method for Edwards' speculation, allowing to transition from zeros of $Z_0(t)$ to those of $Z(t)$ in a continuous manner via a high-dimensional parameter space and proving Theorem \ref{thm:1}. An example of the $730120$-th zero, for which our method is shown to successfully avoid the limitations encountered for Newton's method, is presented in detail in Section \ref{s:ex}. Finally, Section \ref{s:s9} provides a summary of our results, concluding remarks, and potential directions for future research.

\section{The Riemann-Siegel Formula - Numerics vs. Theory} 
\label{s:s2}

In the series of works \cite{HL,HL2,HL3} Hardy and Littlewood developed the formula known as the approximate functional equation. In its most widely used form the formula gives   
\begin{equation} 
\label{eq:HL} 
Z(t) =  2 \sum_{k=1}^{\widetilde{N}(t)} \frac{cos(\theta(t)-ln(k) t) }{\sqrt{k}}+R(t),
\end{equation}  
where $\widetilde{N}(t) = \left [ \sqrt{ \frac{ t}{2 \pi}} \right ]$ and the error term is given by 
\begin{equation} 
R(t)=O \left ( \frac{1}{\sqrt[4]{t}}  \right ).
\end{equation}
This representation, while powerful, for many purposes, requires further refinement of the error term $R(t)$ to enhance the level of precision, a subject not addressed in the original works of Hardy and Littlewood. 

In the 1930s, C. L. Siegel discovered the remarkable Riemann-Siegel formula by examining Riemann's unpublished manuscripts \cite{Si}. While the manuscripts have been studied by various mathematicians, it was Siegel's unique approach to scholarship that allowed him to uncover this crucial aspect of Riemann's work, which had been previously unnoticed. Siegel observed that Riemann's notes contained a variant of the Hardy-Littlewood formula \eqref{eq:HL} but, more astonishingly, also introduced sophisticated saddle-point techniques for the required efficient computation and further asymptotic expansion of the error term \(R(t)\). First, the formula expresses the superb approximation
\begin{equation}
R(t) \sim \frac{e^{-i\theta(t)}e^{-\frac{\pi t}{2}}}{(2\pi)^{\frac{1}{2}+it}e^{-\frac{i\pi}{4}}(1 - ie^{-\pi t})} \int_{L} \frac{e^{-\widetilde{N} z}(-z)^{-\frac{1}{2}+it}}{e^z - 1} dz,
\end{equation}
where the integration path $L$ encompasses the first $2\widetilde{N}(t)$ singularities of the integrand. The path begins at $-\infty$, approaches the saddle point $z = i(2\pi t)^{1/2}$ at an angle of $\pi/4$ at the path of steepest descent, where the integral is "concentrated", and then extending back towards $\infty$. Further analysis leads to the following asymptotic expansion to any $k$-th order 
\begin{equation}
\label{eq:RS-app1}
R(t) \approx (-1)^{\widetilde{N}(t)-1} \left( \frac{t}{2\pi} \right)^{-1/4}  \left( \frac{e^{i\left( \frac{t}{2} ln \left ( \frac{t}{2 \pi} \right ) - \frac{t}{2} - \frac{\pi}{8} - \theta(t) \right)}}{1 - ie^{-\pi t}} \right) \sum_{n=0}^{k} b_n c_n,
\end{equation}
where the $b_n$ coefficients are given by the following recursion relation  
\begin{equation} b_{n+1} = \sqrt{\frac{t}{ 2\pi}} \frac{\pi i (2n + 1)b_n - b_{n-2}}{4\pi^2(n + 1) }, \end{equation}
with $b_0=1$ and $b_{-1}=0$ and the $c_n$ coefficients are given by 
\begin{equation} c_n = \frac{n!}{2^n} \sum_{j=0}^{\left\lfloor n/2 \right\rfloor} \frac{(2\pi i)^j \Psi^{(n-2j)}(p)}{2j! (n - 2j)!}, \end{equation}
where $p= \sqrt{\frac{t}{2 \pi}} - \widetilde{N}(t)$ and $\Psi(p):= \frac{cos ( 2 \pi (p^2-p -\frac{1}{16})) }{cos (2 \pi p ) }$. For instance, expanded to first-order, the Riemann-Siegel formula gives 
\begin{equation} 
R(t) = (-1)^{N(t)-1} \left ( \frac{t}{2 \pi} \right )^{-\frac{1}{4}} \Psi(p)+ O \left ( \frac{1}{t^{\frac{3}{4}}} \right ),
\end{equation}
see \cite{E,Pugh1992}. These innovative techniques allowed Riemann to calculate the first few non-trivial zeros of the zeta function, validating that they lie on the critical line. In fact, since its introduction the Riemann-Siegel formula, especially when enhanced with the Odlyzko-Schönhage algorithm,  has been the main method for numerical verification of the RH, see \cite{Gourdon2004,Le,Odlyzko1992,Od,OdlyzkoSchoenhage1988,Ro,Titchmarsh1935,Turing1953}. The most comprehensive verification to date, achieved by Platt and Trudigian, confirms the RH up to $3\cdot 10^{12}$, see \cite{PT}. However, up to the present day, the Riemann-Siegel formula has primarily been regarded as a numerical tool rather than offering a theoretical explanation for the origins of the RH, for reasons which we will expand on in the following sections. 

In particular, even after the introduction of the Riemann-Siegel formula, many mathematicians maintained that it does not pave a reasonable direction towards proving the RH, nor does it provide a plausible argument for its validity, which is currently entirely considered missing. 
For instance, Edwards, after describing the Riemann-Siegel formula and its proof in detail in his book \cite{E} follows up with a subsection entitled \emph{"Speculations on the Genesis of the Riemann Hypothesis"} (Subsection 7.8), in which he writes:
\begin{quote}
\emph{"Even today, more than a hundred years later, one cannot really give any solid reasons for saying that the truth of the RH is 'probable' etc. Also the verification of the hypothesis for the first three and a half million roots above the real axis perhaps makes it more 'probable'. However, any real \textbf{reason}, any plausibility argument or heuristic basis for the statement, seems entirely lacking."} (H. M. Edwards)
\end{quote}

Similarly, Atle Selberg, which was of course fluent in the theory of computations of zeta, succinctly remarked on the absence of promising approaches for tackling the RH, including attempts to theoretically formalize the content of the Riemann-Siegel formula:
\begin{quote}
\emph{"There have probably been very few attempts at proving the Riemann hypothesis, because, simply, no one has ever had any really good idea for how to go about it!"} (Atle Selberg)
\end{quote}

Nevertheless, in the continuation of the same Section 7.8, as the name suggests, Edwards presents a unique "speculation" on the origin of the Riemann hypothesis. In the next section we describe and analyse Edwards' speculation and discuss its substantial shortcomings when applied to the classical Riemann-Siegel formula.  

 \section{The Original Formulation of Edward's Speculation on the Riemann-Siegel Formula and Related Questions}
\label{s:s3}

In Section 7.6 of his books, Edwards conducts a few sample computations via the Riemann-Siegel formula, including the computation of the first zero of $Z(t)$ as originally conducted by Riemann himself. Edward's speculation, which follows, is essentially the suggestion that the Riemann Hypothesis might have actually originated from hands-on computational observations regarding the inner-workings of the Riemann-Siegel formula. It is encapsulated in the following paragraph:

\begin{center}
\emph{
'My guess is simply that Riemann used the method followed in Section 7.6 to locate roots and that he observed that normally — as long as \(t\) is not so large that the Riemann-Siegel formula contains too many terms and as long as the terms do not exhibit too much reinforcement — this method allows one to go from a zero of the first term \(2\cos(\theta(t))\) to a nearby zero of \(Z(t)\)'} (H. M. Edwards, 1974) 
\end{center}

The computation of the first zero of $Z(t)$  indeed shows that it is in proximity to the first zero of $cos(\theta(t))$. However, one may find that, as presented, the statement remains somewhat vague, almost cryptic, and leaves the reader with various questions. For instance, we have the following three:

\begin{description}

    \item[\textit{Question 1 - The good}] Why should one expect the existence of a general, fundamental relationship between the zeros of $2\cos(\theta(t))$ and those of $Z(t)$, based merely on one numerical computation of the first zero?

    \item [\textit{Question 2 - The bad}] How exactly does one transition from zeros of $2\cos(\theta(t))$ to nearby zeros of $Z(t)$, and what mathematically well-defined method should underpin such a procedure?

    \item [\textit{Question 3 - the ugly}] What is the precise mathematical meaning of "reinforcement" and of having "too many terms" in the Riemann-Siegel formula? 
\end{description}

In the next sections we aim to analyse these questions and show that the vagueness with respect to them is not accidental, each for its own reasons. In particular, we will describe why each of the questions actually has to do with fundamental properties of the $Z$-function and the Riemann-Siegel function. Let us define:

\begin{dfn}[The core of the $Z$-function] We refer to $Z_0(t)=cos(\theta(t))$ as \emph{the core function of $Z(t)$}. 
\end{dfn}

The zeros of the core $Z_0(t)$ are completely well-understood and their description follows immediately from the definition: 

\begin{prop}[Zeros of the core $Z_0(t)$] \label{prop:1} The zeros $t_n \in \mathbb{R}$ of the core $Z_0(t)$ are all real and are given as
the unique solutions of 
\begin{equation} 
\theta(t_n) = \left ( n + \frac{1}{2} \right ) \pi,
\end{equation} 
for any $n \in \mathbb{Z}$.
\end{prop} 

We start with the good part, the expectation for the existence of a general relation between the zeros of the core $Z_0(t)$ and those of $Z(t)$ itself.   

 \section{The Relation of the Zeros of $Z_0(t)$ and the Zeros of $Z(t)$}
\label{s:s4}

Since the idea to transition from a zero of $Z_0(t)$ to a zero of $Z(t)$ is based on the computation of the first zero of $Z(t)$ via the Riemann-Siegel formula, it would be properly described as an experimental observation. As an experimental observation it would greatly benefit from further examples, beyond the one presented by Edwards, for a general pattern to be observed. The following Fig. \ref{fig:f3.1} shows the graphs\footnote{$\ln |F(t)|$ is taken to accentuate the zeros occurring where the function converges to $-\infty$.} of $ln \abs{Z(t)}$ (blue) compared to $ln \abs{Z_0(t)}$ (orange) for $0 \leq t \leq 50$:

\begin{figure}[ht!]
	\centering
		\includegraphics[scale=0.2]{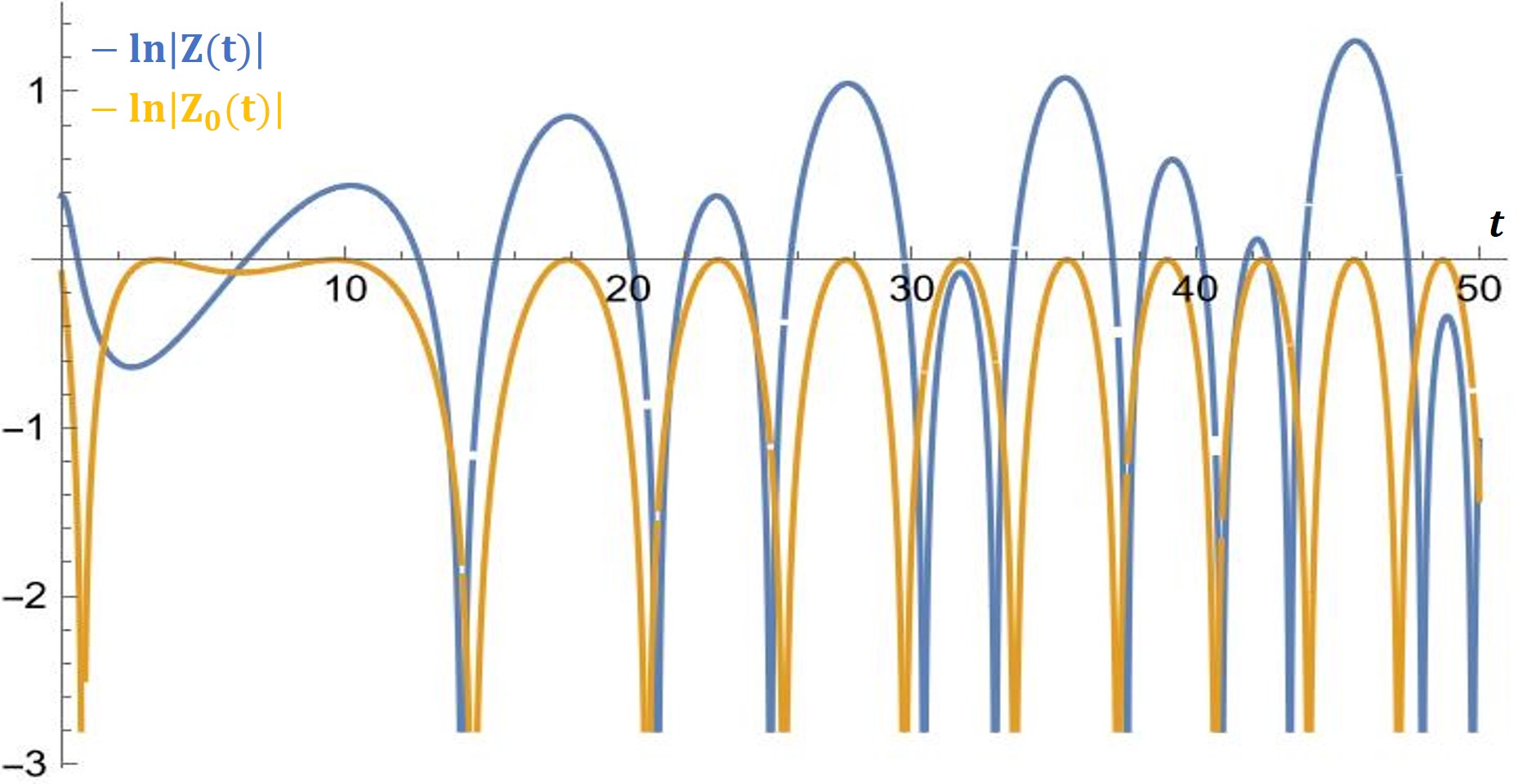} 	
		\caption{\small{Graph of $ln \abs{Z(t)}$ (blue) and $ln \abs{Z_0(t)}$ (orange) for $0 \leq t \leq 50$.}}
\label{fig:f3.1}
	\end{figure}

Figure \ref{fig:f3.1} indeed reveals the remarkable phenomenon Edwards is hinting at, showcasing what might be referred to as one of the most fundamental "facts of nature" regarding the zeros of $Z(t)$. The observed phenomenon is that the zeros of the core $Z_0(t)$, which is the leading term of \eqref{eq:HL}, seem to provide an initial, rough, approximation of the zeros of $Z(t)$.

However, although further vast numerical verification shows that this phenomenon extends well beyond the range presented in Fig. \ref{fig:f3.1}, it remains, fundamentally, strictly an \emph{experimental} phenomenon rooted in empirical evidence rather than a formal mathematical statement. In fact, the crux of Edwards' speculation can be described as the question of whether the observed empirical relationship between the zeros of $Z_0(t)$ and $Z(t)$ can be formalized in a mathematically valid manner, a question of monumental importance for the study of the Riemann Hypothesis.

\begin{rem}[Franca-LeClair] It is reasonable to suggest that the relationship between the zeros of $Z_0(t)$ and $Z(t)$, as an empirical phenomenon, was discovered or even folklorically known, at times, to various mathematicians along the years, and, adopting Edwards' perspective, it was already probably observed by Riemann himself. For instance, Spira \cite{SP1} studied a variant of the core, given by $\zeta_0(s):=1+ \chi(s)$, which in his case arises as the first section of the Hardy-Littlewood formula for $\zeta(s)$, and proved that all of its zeros lie on the critical line. Essentially, the zeros of $\zeta_0(s)$ are given by $s_n = \frac{1}{2}+it_n$ where $t_n$ are the zeros of $Z_0(t)$ described in Proposition \ref{prop:1}.

More recently, the function $1+ \chi(s)$ arose in the investigations of Franca-LeClair \cite{FL,FL2}. Interestingly, in their case, they came to study $1+ \chi(s)$ in a way that is completely independent of the Riemann-Siegel formula, but through their theory of transcendental equations for zeros of zeta. Franca and LeClair also conducted vast further numerical investigations on the relation between zeros of $\zeta_0(s)$ and $\zeta(s)$, uncovering various intriguing additional relations. They also observe that the values of $t_n$ can be expressed in closed form as follows:
\begin{equation} 
\label{eq:FLsol2}
t_n = \frac{(8n-11) \pi}{4 W_0\left( \frac{8n-11}{8e} \right)},
\end{equation}
where $W_0(x)$ represents the principal branch of the Lambert W function, for any $n \in \mathbb{Z}$. It is noteworthy that, although it follows directly from the definition of $W_0(x)$, such a formulation could not have been presented in earlier works. This is because the Lambert $W$-function, despite being historically studied, was only recently formalized as a special function \cite{CGHJK}. 

Interestingly, Franca and LeClair come to a variant of Edwards' question, seeking a method to relate $t_n$ to the zeros of zeta. However, as we will see, their approach encounters challenges similar to those in Edwards' method using the Riemann-Siegel formula, particularly in the ability to describe such a method in practice, beyond an empirical observation. 
\end{rem} 

In order to proceed, the question is, as Edwards points out - how do we transition from $Z_0(t)$ to $Z(t)$ in a mathematically valid manner? Naturally, such a transition requires a formula for computing $Z(t)$, with the Riemann-Siegel formula being the immediate candidate. This introduces the second central issue in Edwards' speculation, which critically prevents it from leading to a plausible argument for the RH.

\section{The Infeasibility of Tracking the Effect of the Terms of the Riemann-Siegel Equation} 
\label{s:s5}
If one wants to proceed to study how zeros of $Z_0(t)$ transition to zeros of $Z(t)$ via the Riemann-Siegel formula, the method is sort of self-evident in theory. We "just" need to track down the contribution of the two components comprising the formula:
 
\begin{description}
\item[\textit{1) The Hardy-Littlewood AFE}] Since $Z_0(t) = \cos(\theta(t))$ is the first term of the Hardy-Littlewood AFE \eqref{eq:HL}, we need to track how the addition of the rest of the terms of the approximate functional equation impacts the location of the zeros. These additional terms are of the form $\frac{\cos(\theta(t) - \ln(k)t)}{\sqrt{k}}$ for $k = 2, \ldots, \widetilde{N}(t)$, where the number of terms, $\widetilde{N}(t) - 1$, increases to infinity as $t$ grows.

\item[\textit{2) The Riemann-Siegel error}]  Assuming the impact of the terms from \eqref{eq:HL} on a given zero \(t_n\) of the core \(Z_0(t)\) can be tracked, we would then need to turn to an analysis of the effect of the Riemann-Siegel error \eqref{eq:RS-app1}. From a numerical point-of-view, this stage is actually less costly than the former, and the coefficients \(b_n\) and \(c_n\) can be efficiently computed using standard symbolic computation software. 

However, from an analytic point-of-view, the "ugly truth" is that a tremendous challenge arises from the fact that \(b_n\) and \(c_n\) lack closed-form expressions. Specifically, \(b_n\) is defined via a recursive formula, and \(c_n\) is reliant on high-order derivatives of \(\Psi(p)\), rendering the task theoretically completely infeasible.
\end{description} 

Hence, from a modern perspective, the main challenge in studying the Riemann Hypothesis arguably stems from the fact that, while the Riemann-Siegel formula enables efficient numerical computations of $Z(t)$ and its zeros, it actually provides minimal analytical insight into their structure. This issue is primarily, but not solely, attributed to the absence of closed-form expressions for the coefficients $c_n$ and $b_n$ involved in estimating the error $R(t)$.

To further illustrate the magnitude of the challenge, it's important to note that for many zeros \(t_n\), the main term of the AFE \eqref{eq:HL} might be sufficiently accurate for observing the RH. However, for certain zeros \(t_n\), the AFE alone falls short of providing adequate verification for the hypothesis. For such zeros the evaluation of the error term \eqref{eq:RS-app1} becomes crucial. The following Fig. \ref{fig:f1.1} displays \( \ln \left| Z(t) \right| \) (orange) alongside the AFE approximation (blue) within the range \(412 < t < 419\):
\begin{figure}[ht!]
\centering
\includegraphics[scale=0.375]{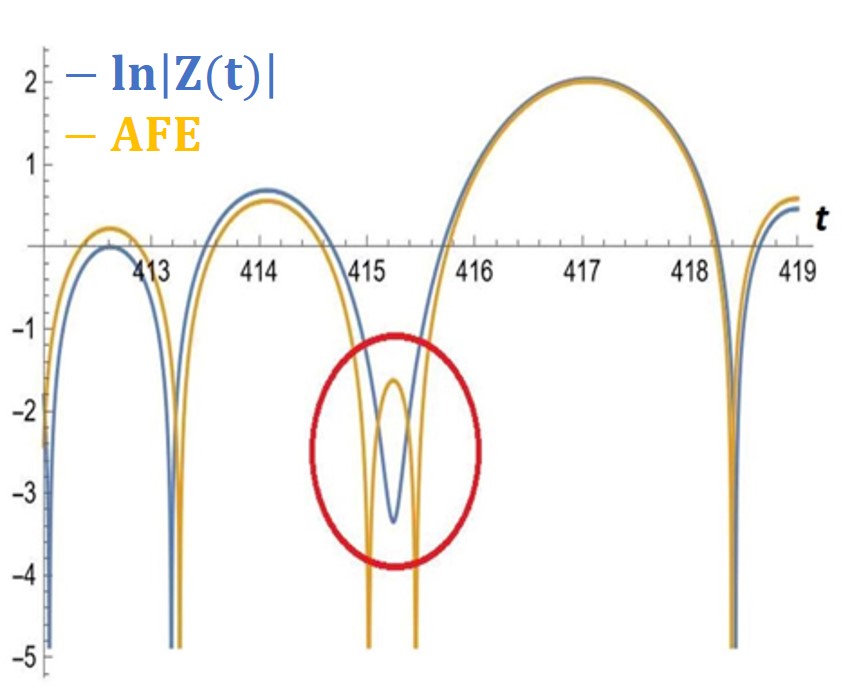}
\caption{\small Graphs of $ln \abs{Z(t)}$ (orange) and the AFE approximation (blue) in the range $412<t<419$}
\label{fig:f1.1}
\end{figure}

 The figure illustrates an instance of two real zeros of $Z(t)$ missed by the Hardy-Littlewood formula (highlighted in red). Actually the AFE formula erroneously identifies these two zeros as complex zeros, and hence is not sensitive enough in order to verify the RH by itself. In fact, It was already observed by Siegel himself \cite{Si} that for any order $k \in \mathbb{N}$ there will be an infinite number of zeros of $Z(t)$ which will require the evaluation of the error term $R(t)$ to an order of at least $k$. In particular, even the question of determining to what order of precision $k=k(n)$ the Riemann-Siegel error term \eqref{eq:RS-app1} must be evaluated to confirm the validity of the RH for the $n$-th zero of $Z(t)$ is completely open. 
 
Hence, while the Riemann-Siegel formula is the main tool for evaluating $Z(t)$ numerically, the characteristics outlined in this section describe what seems to be unbreachable obstacles to any theoretical advancement on the RH through the formula. Specifically, a more manageable alternative to the Riemann-Siegel formula is crucial in order for new analytical progress to be made on the study of zeros of $Z(t)$.

\section{The Failure of the Classical Newton Method to Transition from Zeros of $Z_0(t)$ to Zeros of $Z(t)$} 
\label{s:s6}

Stating the obvious, the RH concerns the zeros of the equation $Z(t) = 0$. In practice, finding closed-form solutions for equations of the form $F(t) = 0$ is rare, and given the transcendental nature of $Z(t)$, there is no expectation of a closed-form expression for its zeros similar to the one given in \eqref{eq:FLsol2} for $Z_0(t)$. 

For example, in the case of algebraic functions where $F(t)$ is a general polynomial, the Abel-Ruffini theorem states that roots can be explicitly solved in closed form only if the degree of $F(t)$ is less than five. However, for higher degrees, the zeros can only be located via numerical methods, such as the Newton method. As a starting point, the method requires a rough initial guess $t=t^1$ for a zero of a function $F(t)$, then applies the iterative process
\begin{equation}
t^{k+1} = t^{k} - \frac{F(t^{k})}{F'(t^{k})},
\end{equation}
aiming to locate a true zero, given by $\widetilde{t} = \lim_{k \rightarrow \infty} (t^k)$, see for instance \cite{DB,SM} for standard references. 

Note that this is similar to the situation we have in the context of Edwards' speculation, where the zeros $t_n$ of the core $Z_0(t)$ are expected to be rough initial guesses of the zeros of $Z(t)$, as discussed in Section \ref{s:s4}. Thus, it is more than natural to suggest the Newton method as a natural candidate for the required method vaguely anticipated by Edwards that would 
\begin{center}
\emph{'...allow one to go from a zero of the first term \( Z_0(t) \) to a nearby zero of \(Z(t)\)'.}
\end{center}
However, if the Newton method or any other, similarly classical, numerical method is indeed the method Edwards has in mind in his speculation, one comes to wonder why does he not specify it directly but rather remain so vague regarding its concrete definition?

In fact, not only does Edwards remain vague regarding the precise definition of the method, after stating the speculation he actually goes to express strong reservations regarding the possible universal applicability of any such a method altogether. In particular, he draws parallels to Gram's law, noting scenarios where it could completely fail. For instance, putting aside the challenges we described in Section \ref{s:s5}, the following cautionary note might have deterred further investigation into the relation between \(Z_0(t)\) and \(Z(t)\):

\begin{center}
\emph{
'The method can certainly fail. To see how completely it can fail, it suffices to consider a complete failure of Gram's law, for example, the failure between \(g_{6708}\) and \(g_{6709}\) in Lehmer's graph.'} (ibid.) 
\end{center}

Nevertheless, it is illuminating to investigate the Newton method as an initial attempt to find the required method for transitioning from zeros of $Z_0(t)$ to zeros of $Z(t)$. Of course, it should be noted that if the Newton method, starting at $t_n \in \mathbb{R}$, would indeed converge to a unique real zero $\widetilde{t}_n \in \mathbb{R} $ of $Z(t)$ for any $n \in \mathbb{Z}$, it would imply the RH.

Let us also note that since our investigation in this section will be based strictly on numerical examples, for reasons that will become clear shortly, the theoretical reservations of Section \ref{s:s5} do not apply. Thus, we can employ the classical Riemann-Siegel formula for this purpose. Consider the following example of the zeros mentioned in Edwards' concern:

\begin{exmp}[The first Lehmer pair of zeros] 
 Let us apply Newton's method for the two Lehmer zeros referred to by Edwards. Starting from $t_{6708}$ and $t_{6709}$, The two corresponding zeros of $Z(t)$ are known to be given approximately by
\begin{equation}
\label{eq:Lehmer}
\begin{array}{ccc}
\widetilde{t}_{6708}\approx 7005.06 & ; & \widetilde{t}_{6709}\approx 7005.10,
\end{array}
\end{equation}
they are characterized as a pair of zeros of $Z(t)$ of unusual close proximity to each other. The first Newton iterations $t^k_{6708}$ and $t^k_{6709}$ for these two zeros are presented in Table \ref{table:1}:

\begin{table}[htbp]
\centering
\begin{tabular}{||c c c ||} 
 \hline
 k & $t^k_{6708}$ & $t^k_{6709}$  \\ [0.5ex] 
 \hline\hline
 1 & 7004.95 & 7005.84  \\ 
 \hline
 2 & 7005.01 & 7005.23 \\
 \hline
 3 & 7005.04 & 7005.15  \\
 \hline
 4 & 7005.05 & 7005.12 \\
 \hline
 5 & 7005.06 & 7005.10 \\
  [1ex] 
 \hline
\end{tabular}
\caption{\small{Newton's Iterations for $t^{k}_{6708}$ and $t^{k}_{6709}$}}
\label{table:1}
\end{table}
The table demonstrates that, despite Edwards' reservations about the possible effect of the unusual proximity of the Lehmer zeros $t_{6708}$ and $t_{6709}$, the iterations $t^k_{6708}$ and $t^k_{6709}$ nevertheless do converge to $\widetilde{t}_{6708}$ and $\widetilde{t}_{6709}$ as given in \eqref{eq:Lehmer}, respectively, without encountering any notable issues in practice.

\end{exmp}

In fact,  experimentation shows that Newton's method tends to work surprisingly well for a vast number of zeros in a similar manner to that presented in the above example. However, despite the apparent success of Newton's method in relating a substantial number of zeros $t_n$ of $Z_0(t)$ to corresponding zeros $\widetilde{t}_n$ of $Z(t)$, our following example reveals that Edwards' original concerned intuition is far from being unjustified, and Newton's method does eventually fail, albeit for far more distant zeros than the ones he originally proposed\footnote{Zeros such as $\widetilde{t}_{730120}$ and $\widetilde{t}_{730121}$ were most certainly beyond the range of numerical verification available to Edwards at the time. In fact, in \cite{E}, he reproduced the graph of $Z(t)$ near the Lehmer zeros originally shown in \cite{Le}. Lehmer notes that this computation required "only" a few hours of machine time on a SWAC computer at the United States National Bureau of Standards, used for this purpose.}:

\begin{exmp}[Failure of Newton's method for the $730120$-th zero] \label{ex:Newton} 
 Consider the $730120$-th zero of $Z(t)$, given by standard zero tables as follows:
$$
 \widetilde{t}_{730121} \approx 450613.8005.
$$ 
Direct computation shows that if we initiate Newton's method from the $730120$-th and $730121$-th zeros of $Z_0(t)$, given by:
$$
\begin{array}{ccc} t_{730120} \approx 450613.9649 & ; & t_{730121} \approx 450614.5269, \end{array}
$$ 
both converge to the $730121$-th zero of $Z(t)$: 
	\begin{equation}
	lim_{k \rightarrow \infty} (t_{730120}^k)=lim_{k \rightarrow \infty} (t_{730121}^k)=\widetilde{t}_{730121}.
	 \end{equation} 
This is visually illustrated in Fig. \ref{fig:Newton} which shows $\ln \abs{Z(t)}$ (blue), $\ln \abs{Z_0(t)}$ (orange) and the locations of the zeros $t_{n-2}$, $t_{n-1}$ and $t_n$ of $Z_0(t)$ for $n=730121$ (purple) in the region $450613.4 \leq t \leq 450614.6$:

	\begin{figure}[ht!]
	\centering
		\includegraphics[scale=0.35]{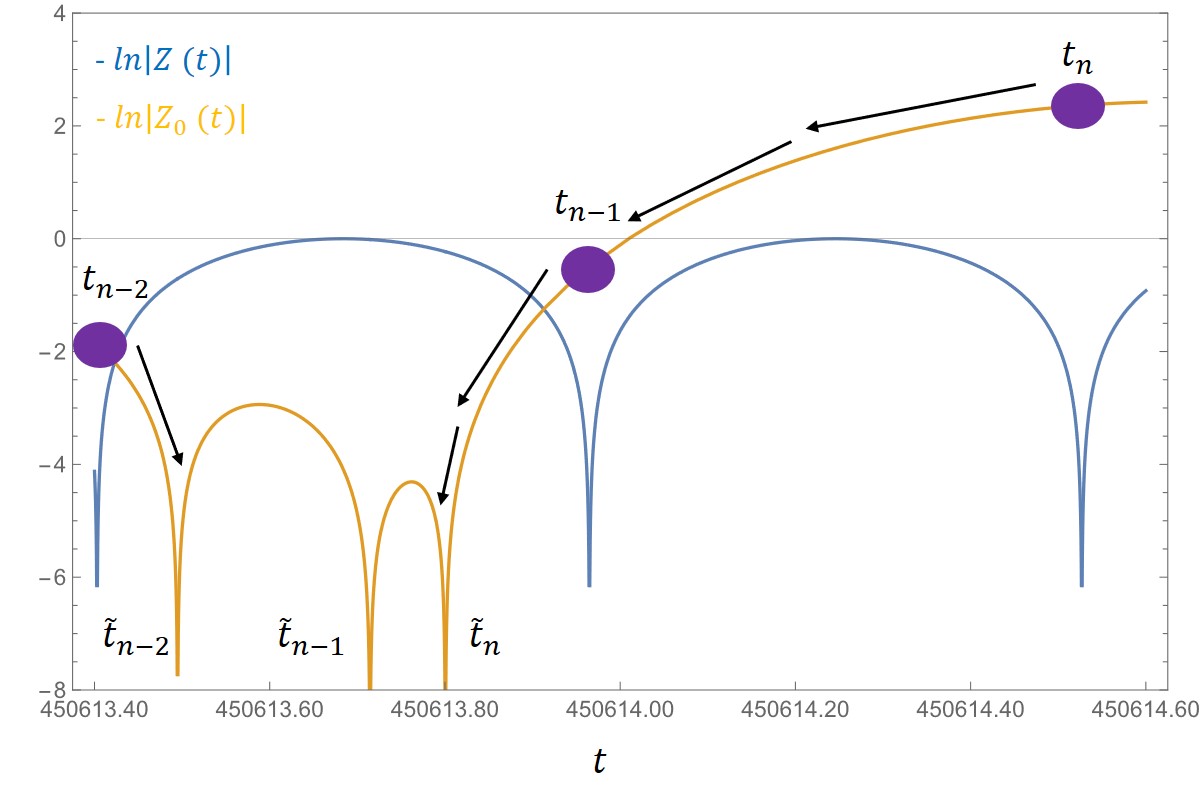} 	
		\caption{\small{Graphs of $\ln \abs{Z(t)}$ (blue), $\ln \abs{Z_0(t)}$ (orange) and the zeros $t_{n-2}$, $t_{n-1}$ and $t_n$ of $Z_0(t)$ for $n=730121$ (purple) in the region $450613.4 \leq t \leq 450614.6$.}}
\label{fig:Newton}
	\end{figure} 	
	
In the setting of Fig. \ref{fig:Newton}, the zeros $t_n$ of $Z_0(t)$ (purple) can be envisioned as balls positioned on the graph of $\ln |Z(t)|$ directly above their corresponding locations on the $t$-axis. Newton's method seeks to identify the zero by effectively rolling these balls down the curve of $\ln |Z(t)|$ until a zero is located (illustrated by black arrows). As demonstrated in Fig. \ref{fig:Newton}, this process results in both $t_{730120}$ and $t_{730121}$ converging to $\widetilde{t}_{730121}$, while $t_{730119}$ converges to $\widetilde{t}_{730119}$, thereby failing to identify $\widetilde{t}_{730120}$ in the middle.

\end{exmp}

In particular, the above example shows that direct, naive, application of Newton's method fails to locate certain zeros of $Z(t)$, and thus  cannot confirm their realness.  It serves as a proof by counterexample of:

\begin{cor}[Failure of Newton's method] Newton's method, starting from $t_n$, the $n$-th zero of the core $Z_0(t)$, does not always converge to $\widetilde{t}_n$, the $n$-th real zero of $Z(t)$.
\end{cor}

So far, our analysis of Edwards' speculation through classical methods identifies two main crucial obstacles: 
\begin{description} 
\item[\textit{The bad}] The direct application of classical methods like Newton's method is generally ineffective and fails for the purpose of locating all zeros, due to their insufficient sensitivity. 
\item[\textit{The ugly}] While the Riemann-Siegel formula is efficient for numerical calculations, for developing theoretical analytical arguments that would be valid for all $n \in \mathbb{Z}$, this formula is completely unfeasible due to the non-closed form of the coefficients $b_n$ and $c_n$ involved in estimating the error term $R(t)$.
\end{description}

These two obstacles demonstrate that, from a classical perspective, the questions raised in Section \ref{s:s3} due to Edwards' supposed vagueness indeed originate from deep, non-trivial issues concerning the zeros of the \(Z\)-function.
 
Hence, to find a well-defined method for relating the zeros of the core \(Z_0(t)\) to those of \(Z(t)\), it becomes mandatory to propose an alternative to the Riemann-Siegel formula. Additionally, this alternative should facilitate a more nuanced procedure than the classical Newton's method for transitioning between zeros.

\section{Spira's high-order sections vs. the Riemann Siegel Formula}
\label{s:s7}
Recall that sections of the Hardy $Z$-function are given by
\begin{equation}
\label{eq:Section} 
Z_N(t) := Z_0(t) +\sum_{k=1}^{N} \frac{\cos(\theta(t)-\ln(k+1) t) }{\sqrt{k+1}}
\end{equation}
for any $N \in \mathbb{N}$, see \cite{J4,SP2,SP1}. These sections are known to approximate the Hardy $Z$-function in two distinct ways. The first, the classical Hardy-Littlewood approximate functional equation \eqref{eq:HL}, is expressed as
\begin{equation}
Z(t)=2Z_{\widetilde{N}(t)}(t) + O \left ( \frac{1}{\sqrt[4]{t}} \right ),
\end{equation}
where $\widetilde{N}(t) = \left [ \sqrt{\frac{t}{2 \pi}} \right ]$. The second, a less conventional approximation to which we refer as Spira's approximation, is
\begin{equation}
\label{eq:Spira}
Z(t)=Z_{N(t)}(t) + O \left ( \frac{1}{\sqrt[4]{t}} \right ),
\end{equation}
with considerably more terms $N(t) = \left [\frac{t}{2} \right ]$. We have explored the remarkable properties of this Spira approximation \eqref{eq:Spira} in \cite{J4}, demonstrating that unlike the classical Hardy-Littlewood AFE \eqref{eq:HL}, which necessitates further correction by the Riemann-Siegel formula, as discussed in Section \ref{s:s2}, the Spira approximation is sufficiently sensitive to discern the RH without additional correction. In particular, based on our previous work on approximate functional equations for series acceleration \cite{J}, we have proven:

\begin{thm}[\cite{J4}] \label{thm:Spira-RH} All the non-trivial zeros of the Hardy $Z$-function $Z(t)$ are real if and only if all the zeros of the high-order Spira approximation $Z_{N(t)}(t)$ are real. 
\end{thm}

For instance, Fig. \ref{fig:f1.4} displays \( \ln \left| Z(t) \right| \) (orange) alongside Spira's high-order approximation \( \ln \left| Z_{N(t)}(t) \right| \)  (blue) within the range \(412 < t < 419\):
\begin{figure}[ht!]
\centering
\includegraphics[scale=0.275]{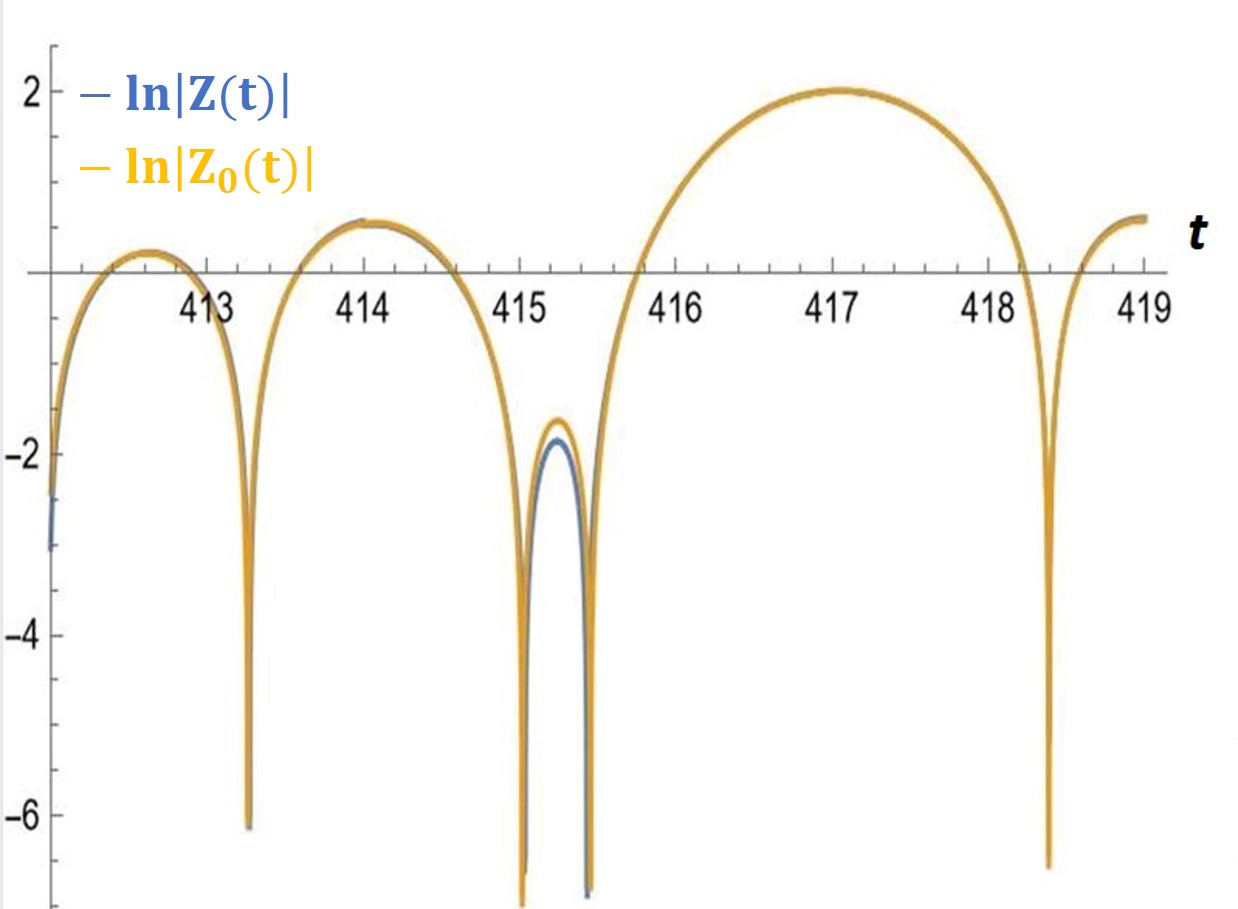}
\caption{\small Graphs of $ln \abs{Z(t)}$ (orange) and Spira's high-order approximation \( \ln \left| Z_{N(t)}(t) \right| \)  (blue) within the range \(412 < t < 419\)}
\label{fig:f1.4}
\end{figure}

Figure \ref{fig:f1.4} shows that indeed Spira's approximation captures the two real zeros missed by the Hardy-Littlewood approximation of Fig \ref{fig:f1.1}. 

\begin{rem} 
The fact that Spira's approximation \eqref{eq:Spira}, whose error is $O \left ( \frac{1}{\sqrt[4]{t}} \right )$, affords better approximation than the Riemann-Siegel formula whose error bound is better than $O \left ( \frac{1}{\sqrt[4]{t}} \right )$ even for the first term, might seem counter-intuitive. Our result in \cite{J4} demonstrates that, in practice, Spira's approximation achieves further improved accuracy beyond the $O \left ( \frac{1}{\sqrt[4]{t}} \right )$ bound obtained for it via classical methods. In fact we show that via acceleration methods that, it converges, as $t$ increases, to an, auxiliary, third approximation of exceptional level of accuracy, $O \left ( e^{- \omega \cdot t}  \right )$, where $\omega > 0$ is a specific positive constant.
\end{rem}

Our objective in the next section is to introduce a new variational method, based on Spira's approximation \eqref{eq:Spira}, aimed at overcoming the obstacles faced by classical techniques of the Riemann-Siegel formula, described so far.

\section{The Variational Approach to Edwards' Speculation} 
\label{s:s8}
The unique relation between the zeros of $Z(t)$ and those of Spira's approximation $Z_{N(t)}$ expressed in Theorem \ref{thm:Spira-RH} enables us to give Edwards' speculation a mathematically well-defined interpretation, in the setting of higher-order sections. Let us introduce the following central definition: 

\begin{dfn}[\( N \)-th $A$-Parameter Space]
For a given \( N \in \mathbb{N} \), we define \( \mathcal{Z}_N \) to be the $A$-parameter space consisting of functions of the form
\[
\label{eq:var} 
Z_N(t; \overline{a}) = Z_0(t) + \sum_{k=1}^{N} \frac{a_k}{\sqrt{k+1}} \cos (\theta(t) - \ln(k+1) t),
\]
where \( \overline{a} = (a_1, \ldots, a_N) \) belongs to \( \mathbb{R}^N \). 
\end{dfn}

In particular, within the $A$-parameter space $\mathcal{Z}_N$ we have two unique elements of special importance: 

\begin{enumerate}
\item The core $Z_0(t) = Z_N (t ; \overline{0})$, at the origin $\overline{a}=\overline{0}$. 

\item The element $Z_N(t; \overline{1})$ at $\overline{a}=\overline{1}$, for which $Z(t) \approx Z_N(t ; \overline{1})$ in $2N \leq t \leq 2N+2$. 
\end{enumerate}

Edwards' speculation now becomes a concrete question, that of studying the way zeros of $Z_N(t; \overline{a})$, within the range $2N \leq t \leq 2N+2$, change along continuous paths connecting $\overline{a}=\overline{0}$ to $\overline{a}=\overline{1}$, attempting to find such paths which keep the zeros real. Figure \ref{fig:f7} shows a schematic illustration of this idea by presenting the core $Z_0(t)  = Z_N (t ; \overline{0})$ connected to $Z_N(t; \overline{1})$ through two different curves in the $A$-parameter space $\mathcal{Z}_N$: 

	 \begin{figure}[ht!]
	\centering
		\includegraphics[scale=0.35]{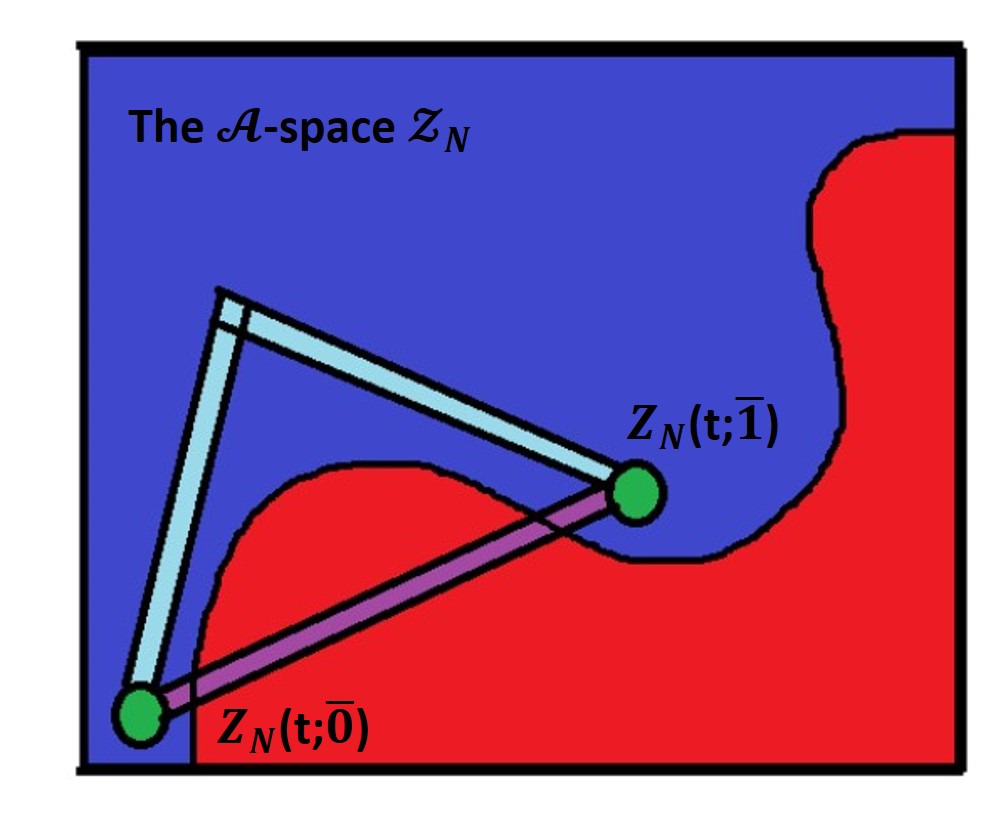} 	
		\caption{\small{Schematic illustration of the core $Z_0(t)  = Z_N (t ; \overline{0})$ connected to $Z_N(t; \overline{1})$ through two different curves in the $A$-parameter space $\mathcal{Z}_N$.}}
\label{fig:f7}
	\end{figure} 

Despite the seeming simplicity of this new variational point-of-view to Edwards' speculation, the following should be noted: 

\begin{enumerate}
\item Within the framework of the classical Riemann-Siegel formula, it is completely impossible to introduce a natural space of the form $\mathcal{Z}_N$, in which both the core $Z_0(t)$ and $Z(t)$ are elements\footnote{\hspace{0.025cm} For simplicity, we refer to $Z(t)$ as an element of $\mathcal{Z}_N$. By this, we specifically mean that Spira's section, $Z_N(t ; \overline{1})$, is within $\mathcal{Z}_N$ and serves, via a reformulation of \eqref{eq:Spira}, as a sufficiently accurate approximation of $Z(t)$ in the relevant range $2N \leq t \leq 2N+2$, for purposes of evaluating the RH.
}, for which Theorem \ref{thm:Spira-RH} is instrumental. 

Essentially, our space $\mathcal{Z}_N$ is derived by incorporating parameters $a_k$ in front of the terms $\frac{\cos(\theta(t) - \ln(k+1) t)}{\sqrt{k+1}}$ in the defining formula of Spira's approximation \eqref{eq:Spira}, and considering $N=N(t)$. Notably, all these terms share a similar structure, naturally depending on the parameter $k$ in a unified manner.

Conversely, within the Riemann-Siegel framework, there is no analogous method for integrating such parameters, primarily because, as detailed in Section \ref{s:s5}, the components used to approximate the error term are not expressed in a closed-form that is analytically manageable.

\item Studying the change of given zeros of $Z_N(t ; \overline{a})$ along continuous paths in the parameter space of $\mathcal{Z}_N$ recasts the RH into a subtle optimization problem, one that is inherently more sensitive than the Newton method. Newton's method, after all, is restricted to only one way for how zeros of $Z_0(t)$ may converge to those of $Z(t)$. Specifically, within the context of the $Z$-function, it tends to identify the closest zero of $Z(t)$ relative to the initially chosen zero $t_n$ of $Z_0(t)$. The allowance for selecting a path introduces significant sensitivity and versatility. We will further elaborate on why such an approach is particularly adept at properly locating the zero of $Z(t)$ naturally corresponding to $t_n$, overcoming the limitations identified for Newton's method.

\end{enumerate}
 
\begin{rem} \label{rem:acc}

The \(A\)-parameter space \(\mathcal{Z}_N\), though not explicitly mentioned, is implicit in \cite{J4}. In fact, the proof of Theorem \ref{thm:Spira-RH} can be viewed as demonstrating the existence of \(Z^{acc}_N(t) = Z_N (t; \overline{a}_N^{acc})\) within \(\mathcal{Z}_N\), referred to as the accelerated approximation satisfying \(Z(t)=Z^{acc}_{N(t)}(t)+O \left( e^{- \omega \cdot t } \right)\) as well as \(\Vert \overline{1} - \overline{a}^{acc} \Vert \rightarrow 0 \) as \(N \rightarrow \infty\), with the distance defined by the natural Euclidean norm in \(\mathcal{Z}_N\).

\end{rem}

Consider a smooth parametrized curve $\gamma(r)$  in $\mathcal{Z}_N$ for $r \in [0, 1]$, connecting the core function $Z_0(t)$, at $r = 0$, to $Z_N (t; \overline{1})$, at $r=1$. Locally, for small enough $r>0$, every zero $t_n$ of $Z_0(t)$ can be smoothly extended to a unique zero $t_n(r)$ of $Z_N(t ; \gamma(r))$. In general, one can always extend the definition of $t_n(r)$ in a continuous, yet not necessarily smooth and unique manner\footnote{\hspace{0.025cm} In analogy, consider the $1$-parametric family $\frac{1}{2}(1-t)(1+t) = r$ within the space of quadrics. For $r = 0$, there are two zeros $t_{\pm} = \pm 1$. For $0 \leq r < \frac{1}{2}$, one can uniquely and smoothly extend $t_{\pm}(r) = \pm \sqrt{1-2r}$. However, at $r = \frac{1}{2}$, the two zeros collide into a double zero at $t = 0$, and subsequently turn into two conjugate complex zeros. Specifically, for $\frac{1}{2} < r \leq 1$, the paths can be continued in a continuous but non-smooth and non-unique manner, necessitating a choice regarding which path proceeds to the upper half and which to the lower half of the complex plane, beyond $r=\frac{1}{2}$.
}. 

A priori, one might expect that the zeros $t_n(r)$ can move arbitrarily in the complex plane, around $t_n$. The following theorem demonstrates that there is a crucial restriction on the zero to, locally, remain real, independent of the curve $\gamma(r)$:

\begin{thm}[Local realness of zeros] \label{thm:A}
Let $\gamma(r)$  be a smooth parametrized curve in $\mathcal{Z}_N$ for $r \in [0, 1]$, connecting the core function $Z_0(t)$, at $r = 0$, to $Z_N (t; \overline{1})$, at $r=1$. Then the zeros \( t_n(r) \) of \( Z_N(t; \gamma(r)) \) remain well-defined, smooth, and real as long as they do not collide with other consecutive zeros, that is, as long as $t_n(r) \neq t_{n  \pm 1}(r)$.
\end{thm}
\begin{proof}
The Riemann-Siegel theta function satisfies $\overline{\theta(\overline{t})} = \theta(t)$. Consequently, all terms of $Z_N(t; \gamma(r))$ satisfy 
\[
\overline{\cos(\theta(\overline{t}) - \ln(k+1) \overline{t})} = \cos(\theta(t) - \ln(k+1) t).
\]
This implies that the set of zeros of \(Z_N(t; \gamma(r))\) is self-conjugate for any \(r \in [0, 1]\). That is, if \(t \in \mathbb{C}\) is a zero of \(Z_N(t; \gamma(r))\), then its complex conjugate \(\overline{t}\) is also a zero. Given that \(t_n(0) = t_n \in \mathbb{R}\) are all real and \(t_n(r)\) varies continuously with \(r\), zeros can only transition off the real line as conjugate pairs in the event of a collision between adjacent zeros.
\end{proof}

Theorem \ref{thm:A} implies that, although elements of \(\mathcal{Z}_N\) have an infinite number of zeros, these zeros exhibit behaviour rather similar to that of polynomials with real coefficients. In particular, Theorem \ref{thm:1} from the introduction follows directly:

\begin{cor}[Edwards' Speculation for High-Order Sections] \label{thm:ESHO}
The Riemann Hypothesis holds if and only if, for any \(n \in \mathbb{Z}\), there exists a path \(\gamma(r)\) in \(\mathcal{Z}_{\left[ \frac{t_n}{2} \right]}\) from \(\overline{a}=\overline{0}\) to \(\overline{a}=\overline{1}\), along which \(t_n\) does not collide with its adjacent zeros \(t_{n\pm 1}\). That is, for which \(t_n(r) \neq t_{n \pm 1}(r)\) for all \(r \in [0,1]\).
\end{cor}

In its essence, Corollary \ref{thm:ESHO} offers a new point-of-view on the RH, 
exhibiting it as a type of non-linear optimization problem. Indeed, it shows that RH 
requires finding a curve $\gamma_n(r)$ connecting $Z_0(t)$ to $Z_N(t)$ in a non-colliding 
way, for any $n \in \mathbb{Z}$. Taking the analogy to polynomials with real coefficients, 
we anticipate the existence of a discriminant function $\Delta_n(\overline{a})$ on the space 
$\mathcal{Z}_N$, such that the hyper-surface $\Delta_n (\overline{a})=0$ is the closure of all 
instances such that $t_n(r)=t_{n+1}(r)$. 

Hence, we are looking to effectively construct the 
curve $\gamma_n(r)$, starting from the core $Z_0(t)$ and incrementally increasing $\overline{a}$ 
from $\overline{a}=\overline{0}$ to $\overline{a}=\overline{1}$ in a manner that minimizes 
the decrease in $\Delta_n(\overline{a})$, avoiding collision, that is vanishing of $\Delta_n(\overline{a})=0$. The main tool for solving such problems is the 
Karush-Kuhn-Tucker optimization theorem, see \cite{Kar,KT}. However, if one is willing to ease 
on optimality, analysis of $\Delta_n (\overline{a})$ might lead to the description of curves 
$\gamma_n(r)$, more natural than the optimal one. This exploration, would require the formal definition and analysis of $\Delta_n(\overline{a})$, which is reserved for future discourse. 

\section{The $A$-variation method - Example for the $730120$-th Zero}
\label{s:ex}

In this section we describe the application of the \(A\)-variation method to the specific case of the \(730120\)-th zero, for which the classical Newton method fails, as shown in Section \ref{s:s6}. While our focus here is on numerical examples, the rationale behind their presentation is twofold. Firstly, they demonstrate the \(A\)-variation method in action for a specific case, providing essential illustration of the concepts discussed thus far. Secondly, we propose that the principles outlined may reflect general, universal properties of the variation of zeros of \(Z(t)\), underscoring the necessity of more comprehensive theoretical investigation.

\subsection{The linear curve} The most natural curve to take in the space $\mathcal{Z}_{N(t)}$ to connect $Z_0(t)$ to $Z_{N(t)}(t ; \overline{1})$ is the linear curve given by 
\begin{equation} 
Z_{N(t)}(t ; \overline{r}) := Z_0(t) +  \sum_{k=1}^{N(t)} \frac{r}{\sqrt{k+1}} cos(\theta(t)-ln(k+1)t), 
\end{equation} 
for $r \in [0,1]$, illustrated in purple in Fig. \ref{fig:f7}. We consider the dynamic evolution of the functions $Z_{N(t)}(t ; \overline{r})$ from $r=0$ to $r=1$ in the range $450613.4 \leq t \leq 450614.6$, within which we have the three zeros $t_{n-2},t_{n-1},t_n$ of $Z_0(t)$ for $n= 730121$.  Given the varied dynamics across different segments of $r \in [0,1]$, we will present each part individually. 

\begin{description}
\item[Step 1] Figure \ref{fig:linear1} shows the graphs of $\ln \abs{Z_{N(t)}(t ; r)}$ for various values within the initial interval $r \in [0,0.2425]$:
	\begin{figure}[ht!]
	\centering
		\includegraphics[scale=0.3]{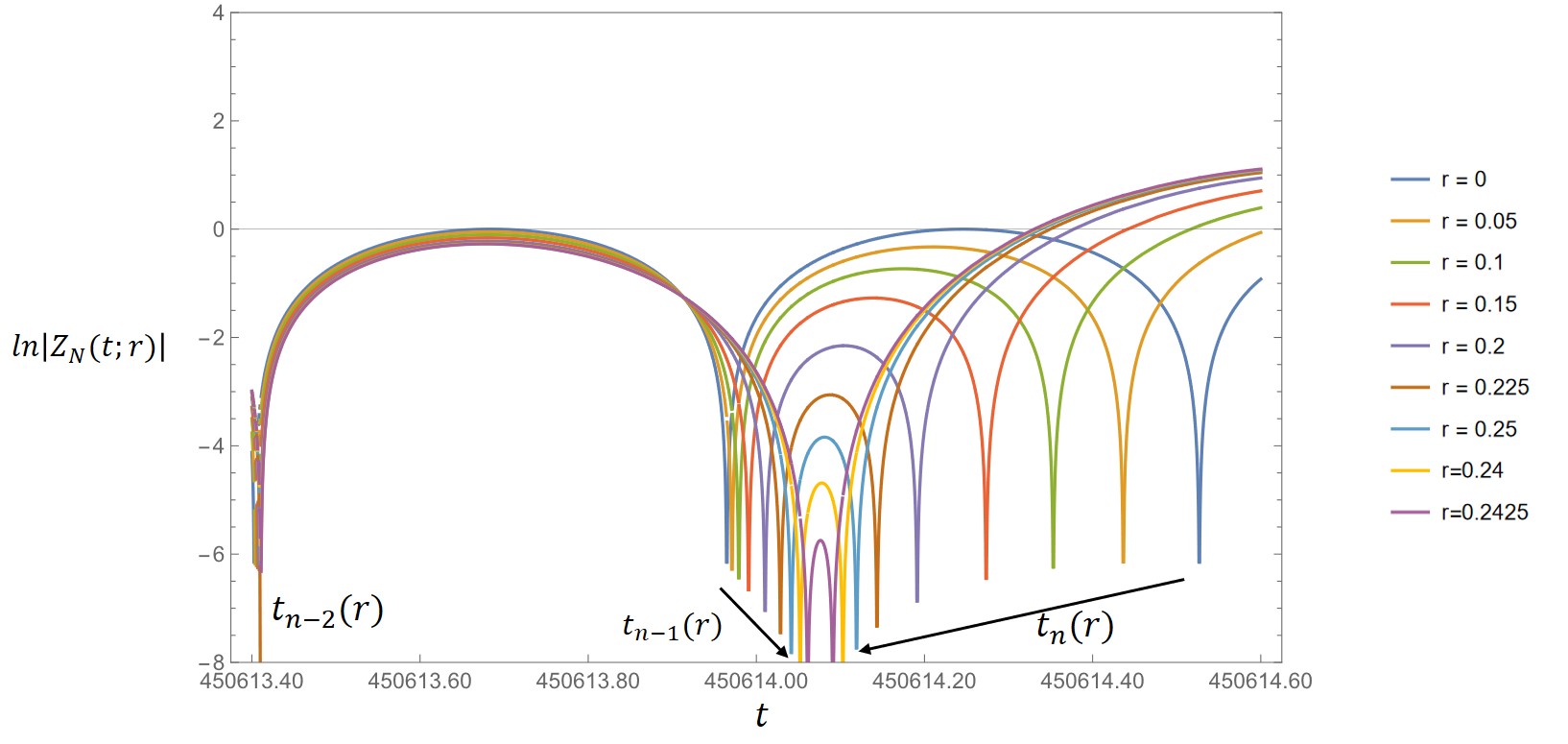} 	
		\caption{\small{$\ln \abs{Z_{N(t)}(t ; r)}$ in $450613.4 \leq t \leq 450614.6$ for  $r \in [0,0.2425]$.}}
\label{fig:linear1}
	\end{figure} 

Within the interval $r \in [0,0.2425]$ the two zeros $t_{730120}(r)$ and $t_{730121}(r)$ are seen to proceed to collide one with the other, while $t_{730119}(r)$ remains relatively static.	

\item[Step 2] Figure \ref{fig:linear2} shows the graphs of $\ln \abs{Z_{N(t)}(t ; r)}$ for various values within the interval $r \in [0.2425,0.85]$, together with $ln \abs{Z_0(t)}$ (blue):
	
	\begin{figure}[ht!]
	\centering
	\hspace{0.4cm}
		\includegraphics[scale=0.3]{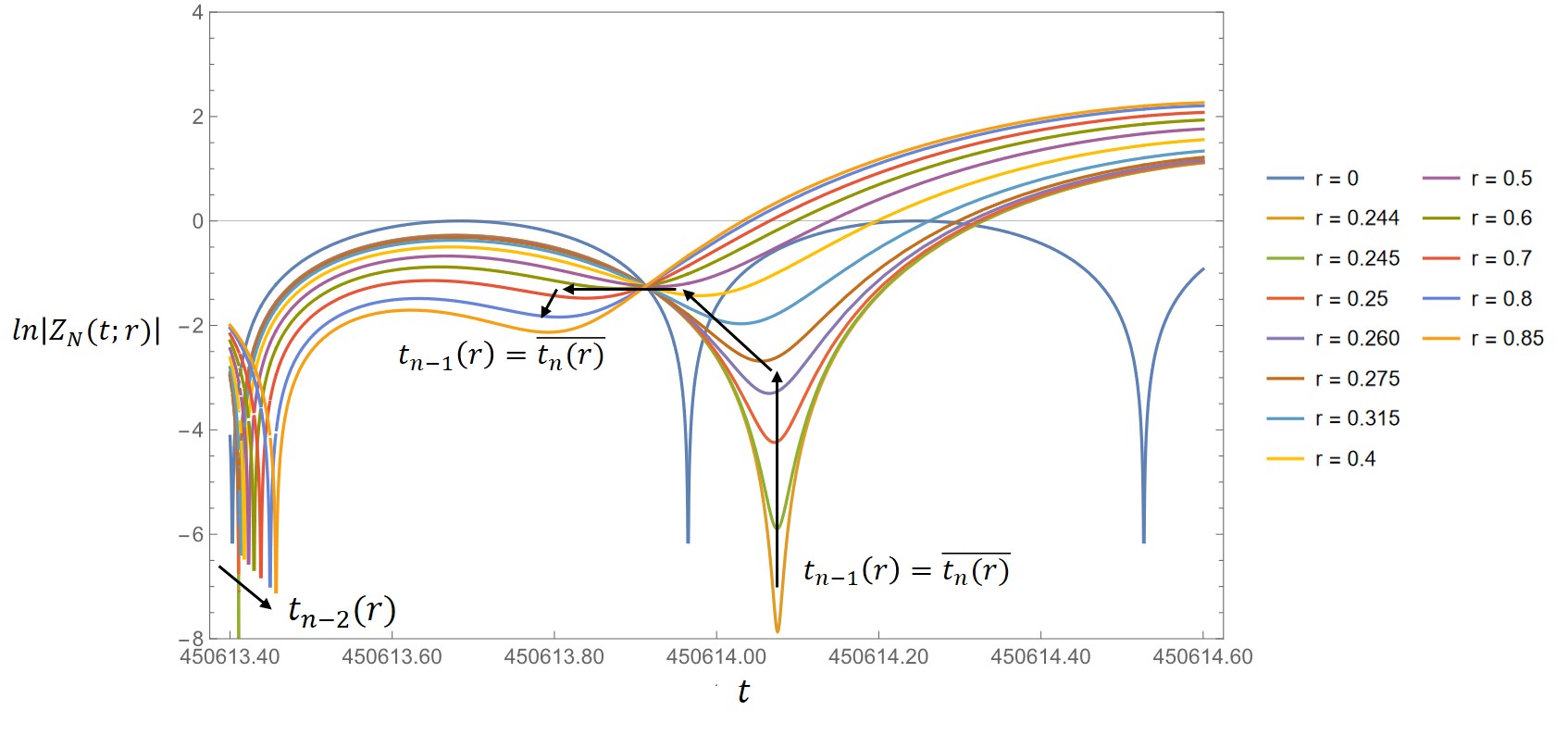} 	
		\caption{\small{$\ln \abs{Z_{N(t)}(t ; r)}$ in $450613.4 \leq t \leq 450614.6$ for  $r \in [0.2425,0.85]$.}}
\label{fig:linear2}
	\end{figure}

	 The figure shows the development of $\ln \abs{Z_{N(t)}(t ; r)}$ after the collision of $t_{730120}(r)$ and $t_{730121}(r)$. In particular, in this region the two corresponding zeros are not real and hence are not presented in the figure, which shows the values only over the real $t$-axis. Indeed, presenting the values $\ln \abs{Z_{N(t)}(t ; r)}$ in a three dimensional plot over the complex plane would have shown that the two corresponding zeros are now actually complex conjugate to each other, as guaranteed in the proof of Theorem \ref{thm:A}, proceeding to move together parallel down to the left of the $t$-axis.

\item[Step 3] Figure \ref{fig:linear3} shows the graphs of $\ln \abs{Z_{N(t)}(t ; r)}$ for various values within the interval $r \in [0.85,0.985]$: 

	\begin{figure}[ht!]
	\centering
		\includegraphics[scale=0.3]{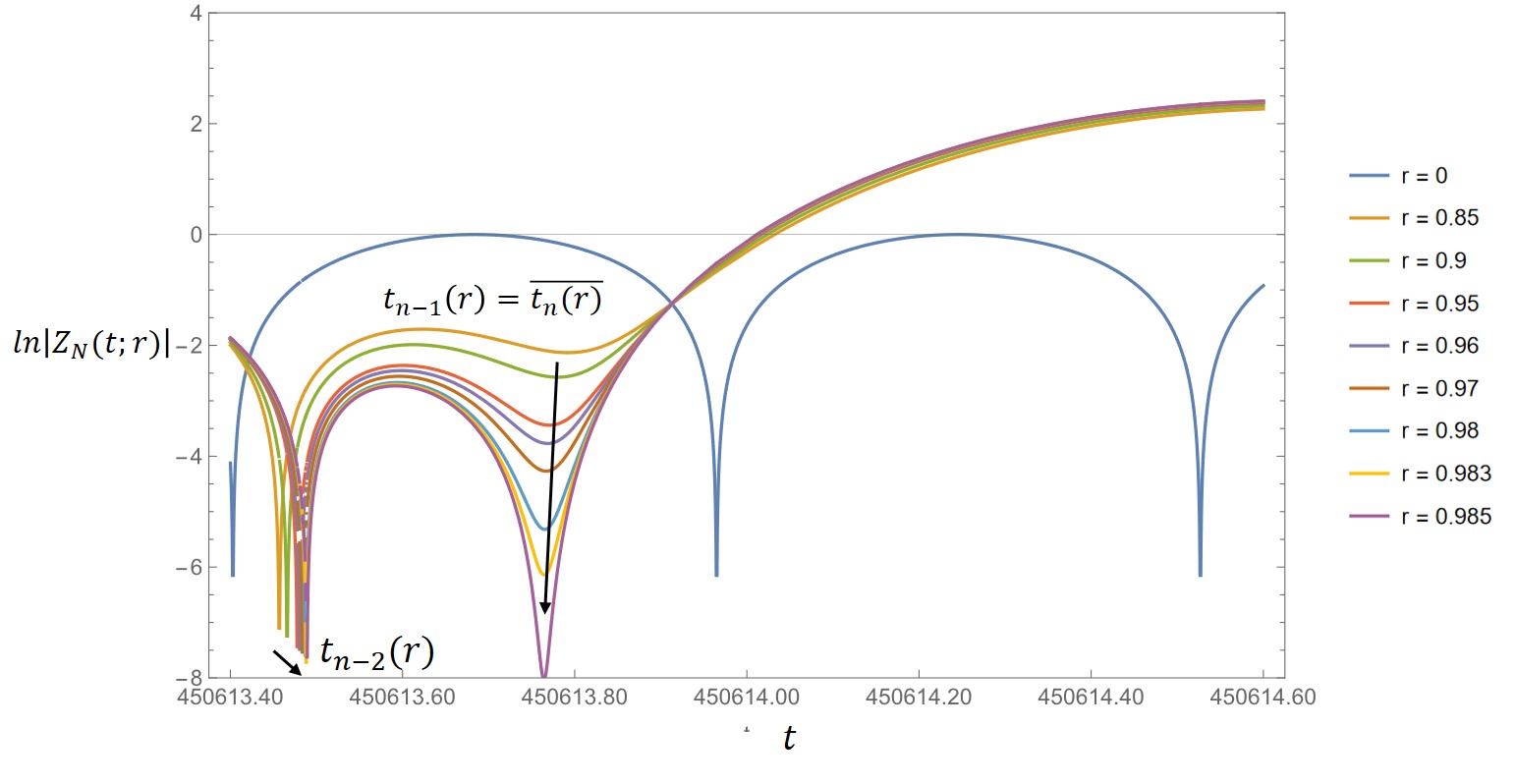} 	
		\caption{\small{$\ln \abs{Z_{N(t)} (t ; r)}$ in $450613.4 \leq t \leq 450614.6$ for $r \in [0.85,0.985]$.}}
\label{fig:linear3}
	\end{figure} 

While in this region the two zeros are still not real, they are starting their return to the real $t$-axis, as the RH obliges. 	
	
\item[Step 4] Figure \ref{fig:linear4} shows the graphs of $\ln \abs{Z_{N(t)}(t ; r)}$ for various values within the final interval $r \in [0.985,1]$:
	
	\begin{figure}[ht!]
	\centering
		\includegraphics[scale=0.3]{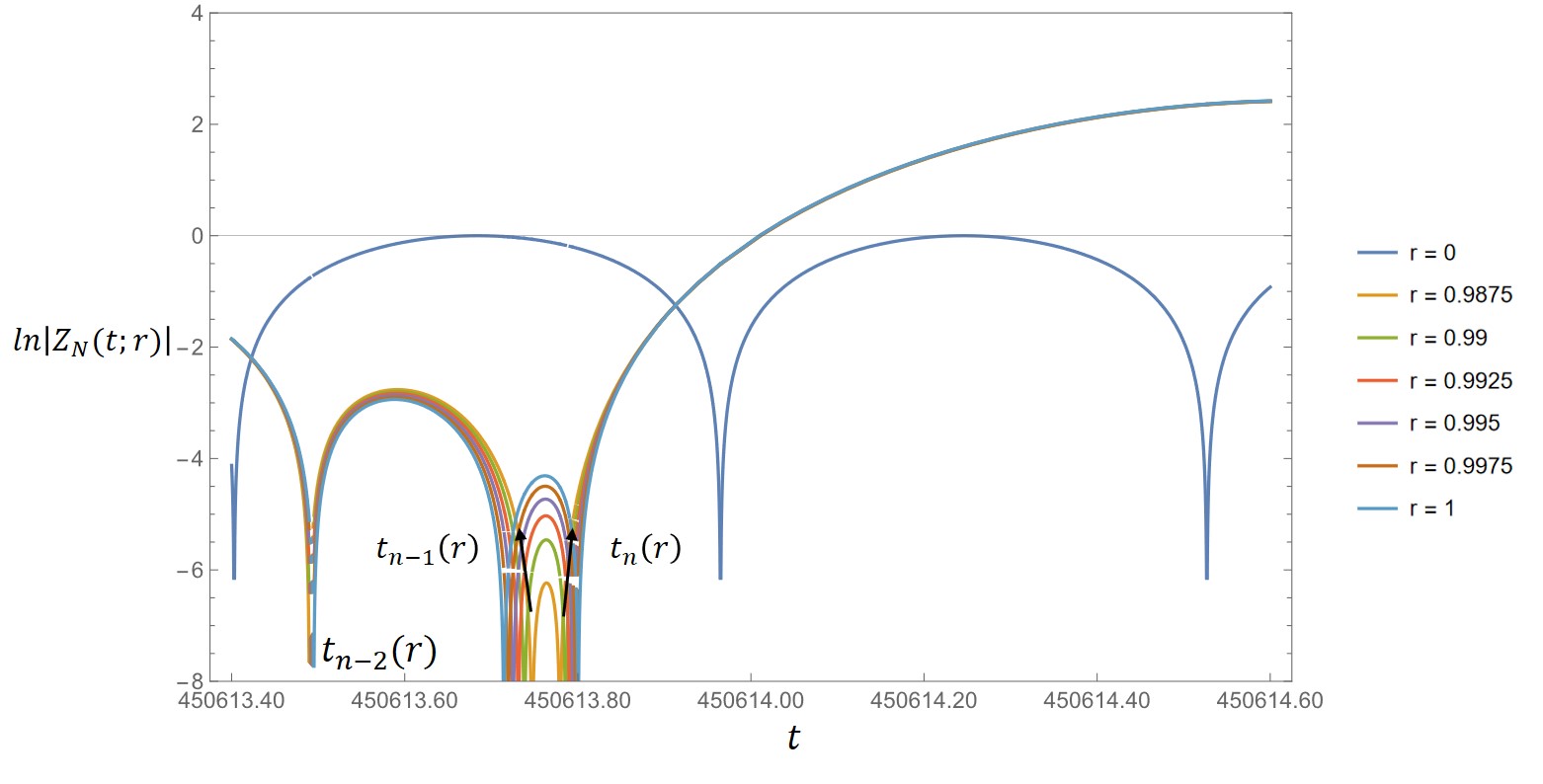} 	
		\caption{\small{$\ln \abs{Z_{N(t)}(t ; r)}$ in $450613.4 \leq t \leq 450614.6$ for  $r \in [0.985,1]$.}}
\label{fig:linear4}
	\end{figure} 
	
	In this region of $r \in [0.985,1]$ the two zeros are real and have returned to the real $t$-axis, they advance to position themselves in their final location as zeros of $Z(t)$, compare Fig. \ref{fig:Newton}. 
\end{description}

The above example of the linear curve $Z_{N(t)}(t; r)$ shows that it already seems to mitigate a few of the problems encountered with the classical Newton method. Indeed, the issue encountered in Example \ref{ex:Newton} was the convergence of both zeros $t_{730120}$ and $t_{730121}$ of $Z_0(t)$ to a single zero $\widetilde{t}_{730121}$ of $Z(t)$, missing $\widetilde{t}_{730120}$. In contrast, tracking the dynamics of $Z_N(t; r)$, we can trace two continuous paths $t_{730120}(r)$ and $t_{730121}(r)$, starting from $t_{730120}$ and $t_{730121}$ at $r=0$ and ending at both $\widetilde{t}_{730120}$ and $\widetilde{t}_{730121}$ for $r=1$, as required. 
 
 However, these paths are not unique due to collisions\footnote{In Fig. \ref{fig:f7}, these collisions are schematically depicted as instances where the linear curve (in purple) transitions from the blue region, where the zeros are real, to the red region, where the zeros are complex conjugate, and vice versa. The non-colliding curve is represented in cyan.} occurring around $r=0.2425$ and $r=0.985$. Given our pursuit for general theoretical principles rather than empirical verification, such colliding paths introduce numerous disadvantages. For instance, while it is feasible to numerically trace the paths of zeros $t_n(r)$ in specific examples, the general scenario raises difficulties as collisions could potentially lead to complex configurations, especially if a pair of conjugate zeros meets another pair from a different collision, complicating theoretical tracking. 
 
Furthermore, and most critically for the RH, if one allows for collisions, it is not clear what general mechanism should guarantee the return of the zeros to the real line at a later time, as indeed observed to occur in the presented example. In fact, it can be argued that theoretical identification of such a mechanism is instrumental for understanding of the underlying origin of the RH.  

\subsection{A non-colliding curve} Due to the limitations outlined above our approach is to try to replace the linear curve by a more sensitive curve $\gamma_n(r)$ within the space $\mathcal{Z}_N$, avoiding collisions altogether. This is where the variational method shows its true power, as the space $\mathcal{Z}_N$ allows for a vast option of curves through which to connect $Z_0(t)$ to $Z_N(t ; \overline{1})$, besides the linear curve. 

In order to describe such a non-colliding curve, let us note that we observe two distinct phenomena occurring in the dynamics of the zeros $t_n(r)$ presented for the linear curve. Firstly, the zeros draw nearer to each other and, secondly, they shift together to the left along the $t$-axis. In particular, we view the collision occurring due to the fact that the shift to the left occurs in an uncorrelated manner with the tendency of the zeros to come in proximity to one another.  We thus aim to construct the non-colliding curve by breaking the linear curve into two separate curves, each responsible for its own task, to occur consecutively one after the other. The first is responsible for achieving the shift to the left while minimizing the relative change in distance between the two zeros, and the second leads to $Z_{N(t)}(t; \overline{1})$, bringing the two zeros closer to their final position. 

We consider the following $2$-parametric system of sections arising from a splitting of the indices \eqref{eq:var}:

\begin{multline} \label{eq:split}
Z_N(t;r_1,r_2):= Z_0(t) +  \sum_{k \in I_{shift}} \frac{r_1}{\sqrt{k+1}} cos(\theta(t)-ln(k+1)t) + \\+  \sum_{k \in I_{collide}}  \frac{r_2}{\sqrt{k+1} }cos(\theta(t)-ln(k+1)t)  \in \mathcal{Z}_N,
\end{multline} 
where we take the shifting indices to be 
\begin{equation} 
I_{shift}= \left \{ 1,2,4,6,12 \right \}, 
\end{equation} 
and the colliding set of indices $I_{collide}$ is taken to be the complement of $I_{shift}$. In this case, the split of indices into shifting and colliding has been obtained experimentally. 

\begin{description}
\item[Shifting stage] 
	Figure \ref{fig:noncol1} shows graphs of $\ln \abs{Z_{N(t)}(t ; r_1 ,r_2)}$ for the following values of $(r_1,r_2)$: $(0,0)$ (blue), $(0.25,0.05)$ (yellow), $(0.55,0.15)$ (green), (0.75,0.25) (red), $(1,0.41)$ (purple) in the range $450613.58 \leq t \leq 450614.8$: 
	
	\begin{figure}[ht!]
	\centering
		\includegraphics[scale=0.3]{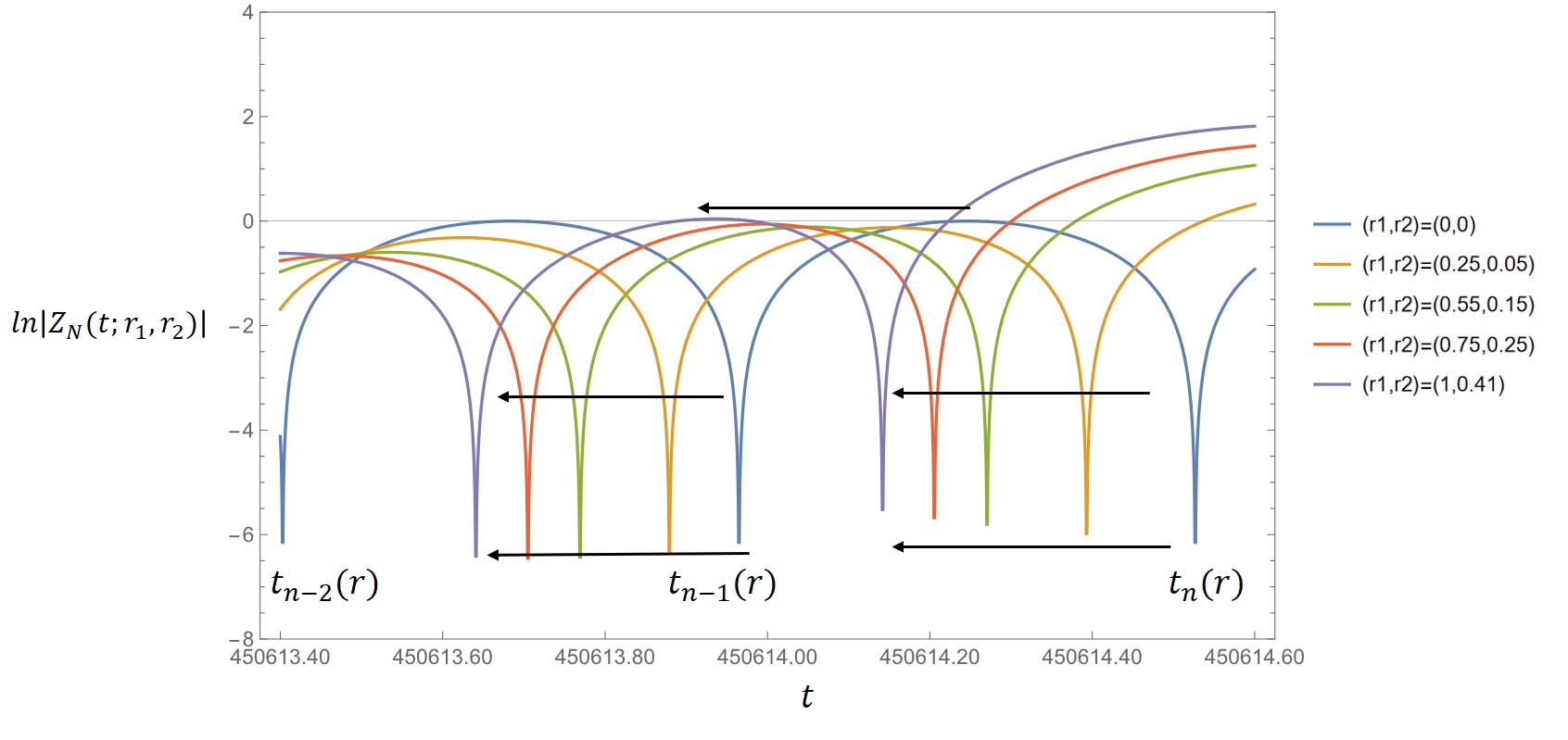} 	
		\caption{\small{Graphs of $\ln \abs{Z_{N(t)}(t ; r_1 ,r_2)}$ along the following values of $(r_1,r_2)$:  $(0,0)$ (blue), $(0.25,0.05)$ (yellow), $(0.55,0.15)$ (green), (0.75,0.25) (red), $(1,0.41)$ (purple) in the range $450613.58 \leq t \leq 450614.8$.}}
\label{fig:noncol1}
	\end{figure} 

We see that as $(r_1,r_2)$ transition along the outlined curve from $(0,0)$ (blue) to $(1,0.41)$ (purple) the zeros $t_{n}(r)$ and $t_{n-1}(r)$ transition continuously to the left in a way that generally preserves their relative distance. In particular, contrary to the case of the linear curve, this transition to the left is now organized in a way that avoids the un-required collision.  

\item[Descending stage] 
	Figure \ref{fig:des} shows the graphs of $\ln \abs{Z_{N(t)}(t ; r_1 ,r_2)}$ along the following points $(r_1,r_2)$:  $(1,0.41)$ (blue), $(1,0.6)$ (yellow), $(1,0.8)$ (green), (1,0.95) (red), $(1,1)$ (purple) in the range $450613.58 \leq t \leq 450614.8$:

	\begin{figure}[ht!]
	\centering
		\includegraphics[scale=0.3]{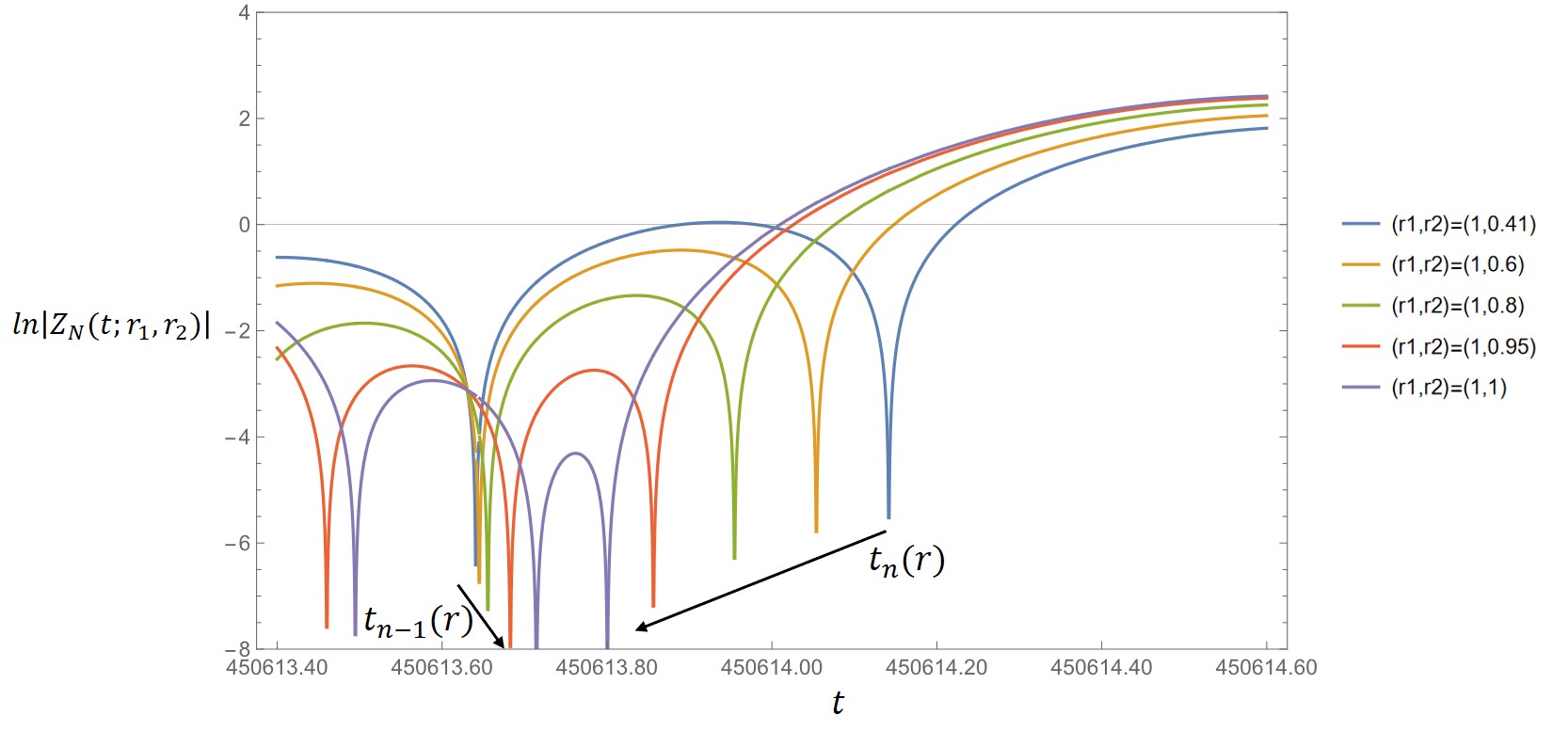} 	
		\caption{\small{Graphs of $\ln \abs{Z_{N(t)}(t ; r_1 ,r_2)}$ along the following points $(r_1,r_2)$:  $(1,0.41)$ (blue), $(1,0.6)$ (yellow), $(1,0.8)$ (green), (1,0.95) (red), $(1,1)$ (purple) in the range $450613.58 \leq t \leq 450614.8$.}}
\label{fig:des}
	\end{figure} 
	\end{description}
	
	As required, collisions have been avoided altogether in the transition from $Z_0(t)$ to $Z_N(t; \overline{1})$ achieved by concatenating the shifting and descending stages described above. In particular, the paths of zeros $t_{730120}(r)$ and $t_{730121}(r)$ are defined in a unique and continuous way, starting from $t_{730120}$ and $t_{730121}$ and terminating at $\widetilde{t}_{730120}(r)$ and $\widetilde{t}_{730121}(r)$, respectively, in a way that identifies the realness of both zeros. 
	
	The overall process can be conceptualized by viewing the shifting indices as utilizing their potential for a collision to instead shift the zeros to the left, effectively depleting the energy required for further collisions. On the other hand, in situations where the linear curve inherently avoids collisions, the shifting component's contribution tends to be minor.

	In this example of $n=730121$, the split of indices into shifting and descending indices has been mainly obtained experimentally, as mentioned. Justifying the existence of a similar split for any $n \in \mathbb{Z}$, where the linear curve requires correction, even if just through a heuristic plausibility argument, would greatly enhance our understanding of the RH. This may necessitate the discovery of new fundamental properties of the function $Z(t)$. We propose that a broad theoretical exploration of such traits would require in-depth study of the geometric aspects of the proposed discriminant $\Delta_n(\overline{a})$.
	
\begin{rem}[The \(A\)-variation method and Newton's method] In fact, the role of the Newton method is not entirely negated within the \(A\)-variation method framework. In practice, to trace the path of zeros \(t_n(r)\), it is still necessary to segment \(r \in [0,1]\) into smaller intervals \(0 < r_1 < r_2 < \ldots < r_m = 1\), and inductively find \(t_n(r_{j+1})\) as the nearest zero of \(Z_{N(t)}(t; \gamma (r_{j+1}))\) to \(t_n(r_j)\) via the classical Newton method. Thus, the failure of Newton's method, as discussed in Section \ref{s:s6}, primarily concerns its direct application from \(Z_0(t)\) to \(Z(t)\) in a single leap. The \(A\)-variation method, however, allows to facilitate a connection from \(Z_0(t)\) to \(Z(t)\) through more gradual steps, potentially ensuring that the transition at each stage uniquely leads to real zeros.
\end{rem}
 
 \section{Summary and concluding remarks} \label{s:s9} 
 
 The incremental advancements that have been achieved on the RH since its introduction by Riemann in 1859  owe mainly to a tradition of scholarship based on engagement with previous historical mathematical studies on the subject, as well as the exploration of methods for the direct computation of zeta on the critical line, both approaches leading back to Riemann's original work itself. For instance, the discovery of the Riemann-Siegel formula was a direct consequence of Siegel's thorough examination of Riemann's original, unpublished, computational manuscripts. Building on this tradition, Edwards' speculation, through his detailed analysis of the inner-workings of the Riemann-Siegel formula, reminiscent of Siegel's approach to Riemann's work, inspired contemporary questions based on insights that may have been envisioned by Riemann himself. 
 
 At the foundation of Edwards' speculation is the notion that a conceptual, mathematically well-defined link could be established between the zeros of the core \(Z_0(t)\) and those of \(Z(t)\). However, following the introduction of his speculation, Edwards turns to express strong reservations about the practical application of such an approach, while offering only broad and somewhat vague reasons for his doubts.
 
 In the first part of this work, we have systematically analysed Edwards' speculation, identifying in detail two main significant obstacles that hinder its practical application within the original context of the Riemann-Siegel formula. Firstly, we encounter the insurmountable challenge of accounting for the error terms inherent in the Riemann-Siegel formula. Secondly, there is the ambiguity surrounding the specific methodology required to effectively link the core zeros with those of \(Z(t)\). In the process we have showed the inadequacy of methods such as the classical Newton's method in this context.
 
 In the second part of this work, we have adapted Edwards' speculation to the context of Spira's high-order sections, a setting in which, remarkably, most of the obstacles outlined in the first part for the Riemann-Siegel setting are avoided. Spira conducted experimental studies on methods for the direct computation of zeta on the critical line \cite{SP2,SP1}. In his studies, Spira observed that sections $Z_N(t)$ of $Z(t)$ defined in \eqref{eq:Section} could approximate $Z(t)$ in two distinct manners:
\begin{description}  
  \item[(1) Hardy-Littlewood AFE] $Z(t) \approx 2Z_{\widetilde{N}(t)}(t)$ for $\widetilde{N}(t) = \left[ \sqrt{ \frac{t}{2 \pi}} \right]$.
  
  \item[(2) High-order approximation] $Z(t) \approx Z_{N(t)}$ for $N(t)=\left[ \frac{t}{2} \right]$.
  \end{description}

The Hardy-Littlewood AFE is highly classical and lies in the basis of the Riemann-Siegel formula. In comparison, the high-order approximation appears inferior to the AFE in almost all aspects. It is significantly more computationally demanding and, based on standard asymptotic analysis, expected to be less efficient. Nonetheless, through his empirical investigations, Spira suggested that, in contrast to the AFE, high-order approximations possess the sensitivity required to observe the RH directly, eliminating the need for further correction. In \cite{J4}, we have theoretically validated Spira's observations by proving that the RH for Spira's high-order sections is, indeed, equivalent to the RH for \(Z(t)\).

Based on these results, in the current work we have introduced the $A$-variation space $\mathcal{Z}_N$ whose elements are given by 

\begin{equation} 
Z_N(t; \overline{a}) = Z_0(t) + \sum_{k=1}^{N} \frac{a_k}{\sqrt{k+1} } cos ( \theta (t) - ln(k+1) t)
\end{equation}
where $\overline{a}=(a_1,...,a_N) \in \mathbb{R}^N$. This allowed us to prove Theorem \ref{thm:1} which expresses Edwards' speculation, and hence the RH, as an optimization problem requiring the study of the way zeros of \(Z_{N(t)}(t; \overline{a})\) transition from \(Z_0(t)\) to \(Z_{N(t)}(t)\) along various paths in \(\mathcal{Z}_N\) from \(\overline{a}=\overline{0}\) to \(\overline{a}=\overline{1}\), with the aim of finding such paths which avoid collisions between consecutive zeros for any $n \in \mathbb{Z}$.

This new variational viewpoint on the RH leads to several immediate foundational questions. Our initial analysis of the space \(\mathcal{Z}_N\) reveals that the zeros of its elements, despite being infinite in number, exhibit certain characteristics reminiscent of polynomials with real coefficients. Intriguingly, our findings suggest that, akin to polynomials with real coefficients, the space \(\mathcal{Z}_N\) might also be endowed with a well-defined discriminant hypersurface, with the RH deeply intertwined with the geometry of this discriminant. This promising direction is something we aim to delve into in future work.

It should be noted that our approach, while according to Edwards is deeply rooted in Riemann’s own 19th-century investigations, carries a 21st-century flair. Indeed, our strategy of tracing zeros along specific paths within the vast-dimensional space \(\mathcal{Z}_N\), coupled with our optimization approach to the RH, shares similarities with non-linear optimization problems encountered in modern applications. For instance, the use of gradient descent techniques in high-dimensional spaces for training machine learning algorithms \cite{Goodfellow-et-al-2016}. This is not to imply that theoretical exploration of such ideas was previously impossible. However, the capacity for practical investigation of such large dimensional spaces and the non-linear optimization problems defined therein was not readily accessible. For example, while Spira conducted extensive and pioneering numerical investigations of sections of \(Z(t)\) in the 60s and 70s, one might argue that technological limitations, including those related to data visualization, may have prevented him from viewing the problem through a variational lens, studying the zeros of sections \(Z_N(t)\) for given fixed $N \in \mathbb{N}$ as individual objects rather than their dynamic change.

In fact, this intersection of our \(A\)-variation approach with techniques intrinsic to machine learning not only underscores general theoretical similarities but now also opens up new potential venues for studying the RH. Indeed, our \(A\)-variation approach to the RH involves identifying non-colliding paths in \(\mathcal{Z}_N\) for any \(n \in \mathbb{Z}\). In Section \ref{s:ex}, we successfully identified such a non-colliding curve for the specific example of \(n=730120\) through experimentation, pinpointing the shifting parameters by direct methods. The exploration of \(\mathcal{Z}_N\) with machine learning algorithms to discover patterns that can facilitate the identification of such non-colliding paths and their geometry in the general case heralds an intriguing direction for future research. The potential application of these techniques might represent a novel opportunity for contemporary examination of the RH, creating a bridge between classical mathematical conjectures and insights, following Riemann's original explorations, to cutting-edge modern optimization concepts.

\section*{Declarations}

\subsection*{Funding}
No funding was received to assist with the preparation of this manuscript.

\subsection*{Conflicts of Interest/Competing Interests}
The authors declare that they have no conflict of interest or competing interests relevant to the content of this article.

\subsection*{Data Availability}
The authors declare that the data supporting the findings of this study are available within the paper or from the corresponding author upon reasonable request.

\bibliographystyle{plain} 
\bibliography{ES.bib} 

\end{document}